\documentclass[12pt,reqno]{amsart}
\usepackage[left=1in,right=1in,top=1in,bottom=1in]{geometry}
\usepackage{amssymb}
\usepackage{amsmath}
\usepackage{mathrsfs}
\usepackage{amsfonts}
\usepackage{hyperref}
\usepackage{xcolor}
\usepackage{bbm}
\usepackage{tikz}

\title[Stable-limit partially symmetric Macdonald functions]{Stable-limit partially symmetric Macdonald functions and parabolic flag Hilbert schemes}
\author{Daniel Orr}
\address{Department of Mathematics (0123),
460 McBryde Hall, Virginia Tech,
225 Stanger Street,
Blacksburg, VA 24061-1026}
\email{dorr@vt.edu}
\author{Milo Bechtloff Weising}
\address{Department of Mathematics (0123),
460 McBryde Hall, Virginia Tech,
225 Stanger Street,
Blacksburg, VA 24061-1026}
\email{milojbw@vt.edu}
\date{\today}

\theoremstyle{plain}
\newtheorem{thm}{Theorem}[section]
\newtheorem*{thm-main}{Main Theorem}
\newtheorem{defn}[thm]{Definition}
\newtheorem{cor}[thm]{Corollary}
\newtheorem{lem}[thm]{Lemma}
\newtheorem{example}[thm]{Example}
\newtheorem{prop}[thm]{Proposition}
\newtheorem{conj}[thm]{Conjecture}
\newtheorem{rem}[thm]{Remark}

\newcommand{\la}{\lambda}
\newcommand{\ga}{\gamma}

\newcommand{\A}{\mathbb{A}}
\newcommand{\B}{\mathbb{B}}
\newcommand{\C}{\mathbb{C}}
\newcommand{\K}{\mathbb{K}}
\newcommand{\T}{\mathbb{T}}
\newcommand{\Q}{\mathbb{Q}}
\newcommand{\Y}{\mathbb{Y}}
\newcommand{\Z}{\mathbb{Z}}

\newcommand{\PFH}{\mathrm{PFH}}
\newcommand{\Hilb}{\mathrm{Hilb}}
\newcommand{\sort}{\mathrm{sort}}

\newcommand{\id}{\mathsf{e}}
\newcommand{\dd}{\mathsf{d}}
\newcommand{\TT}{\mathsf{T}}
\newcommand{\yy}{\mathsf{y}}
\newcommand{\zz}{\mathsf{z}}

\newcommand{\sh}{\mathrm{sh}}

\newcommand{\dg}{\mathrm{dg}}
\newcommand{\ta}{\tilde{a}}

\newcommand{\fS}{\mathfrak{S}}

\usepackage[
    backend=bibtex,
    maxnames=5,
    style=alphabetic
  ]{biblatex}
\begin{document}

\begin{abstract}
The modified Macdonald functions $\widetilde{H}_{\mu}$ are fundamental objects in modern algebraic combinatorics. Haiman showed that there is a correspondence between the $(\mathbb{C}^{*})^2$-fixed points $I_{\mu}$ of the Hilbert schemes $\mathrm{Hilb}_{n}(\mathbb{C}^2)$ and the functions $\widetilde{H}_{\mu}$ realizing a derived equivalence between $(\mathbb{C}^{*})^2$-equivariant coherent sheaves on $\mathrm{Hilb}_{n}(\mathbb{C}^2)$ and $(\mathfrak{S}_n \times (\mathbb{C}^{*})^2)$-equivariant coherent sheaves on $(\mathbb{C}^2)^n.$ Carlsson--Gorsky--Mellit introduced a larger family of smooth projective varieties $\mathrm{PFH}_{n,n-k}$ called the parabolic flag Hilbert schemes. They showed that an algebra $\mathbb{B}_{q,t}$, directly related to the double Dyck path algebra $\mathbb{A}_{q,t}$ employed in Carlsson--Mellit's proof of the Shuffle Theorem, acts naturally on the $(\mathbb{C}^{*})^2$-equivariant K-theory $U_{\bullet}$ of these spaces and, moreover, there is a $\B_{q,t}$-isomorphism $\Phi: U_{\bullet} \rightarrow V_{\bullet}$ where $V_{\bullet}$ is the polynomial representation. The isomorphism $\Phi: U_{\bullet} \rightarrow V_{\bullet}$ is known to extend Haiman's correspondence. In this paper, we explicitly compute the images $\Phi(H_{\mu,w})$ of the normalized $(\mathbb{C}^{*})^2$-fixed point classes $H_{\mu,w}$ of the spaces $\mathrm{PFH}_{n,n-k}$ and show they agree with the modified partially symmetric Macdonald polynomials $\widetilde{H}_{(\lambda|\gamma)}$ introduced by Goodberry-Orr, confirming their prior conjecture. We use this result to give an explicit formula for the action of the involution $\mathcal{N}$ on $V_{\bullet}.$
\end{abstract}

\maketitle

\tableofcontents

\section{Introduction}

In \cite{GO}, the first named author and Goodberry formulated a conjecture describing the images of fixed-point classes under an isomorphism
\begin{align}\label{E:Phi}
\bigoplus_{k\ge 0}\bigoplus_{n \ge k} K_\T(\PFH_{n,n-k})_\mathrm{loc} \overset{\Phi}{\longrightarrow} \bigoplus_{k\ge 0}\Lambda\otimes\K[y_1,\dotsc,y_k]
\end{align}
arising from the work of Carlsson, Gorsky and Mellit \cite{GCM_2017}. Here, $\PFH_{n,n-k}$ is the parabolic flag Hilbert scheme of \cite{GCM_2017}, which is acted upon naturally by a torus $\T=(\C^\times)^2$, and $K_{\T}(\PFH_{n,n-k})_{\mathrm{loc}}$ denotes its localized torus-equivariant $K$-theory. The isomorphism \eqref{E:Phi} is one of $\B_{t,q}$-representations, where $\B_{t,q}$ is a cousin of the Carlsson--Mellit algebra $\A_{t,q}$ introduced in \cite{CM_2015}. On the other side of the isomorphism $\Phi$, $\Lambda$ is the algebra of symmetric functions over the field $\K=\Q(q,t)=\mathrm{Frac}(K_\T(\mathrm{pt}))$ in variables $x_1,x_2,\dotsc$, and the total space $\bigoplus_{k\ge 0}\Lambda\otimes\K[y_1,\dotsc,y_k]$ affords the polynomial representation of $\B_{t,q}$.

We recall that $\Phi$ respects the index $k$ on both sides and that for $k=0$ one obtains an isomorphism
\begin{align}\label{E:Phi-0}
\bigoplus_{n\ge 0} K_\T(\Hilb_n)_\mathrm{loc} \overset{\Phi_0}{\longrightarrow} \Lambda
\end{align}
where $\Hilb_n$ is the Hilbert scheme of $n$ points in the plane. It is shown in \cite{GCM_2017} that $\Phi_0$ agrees up to some explicit scalars with Haiman's geometric realization of modified Macdonald symmetric functions. The conjecture of \cite{GO} extends this realization to the realm of partially-symmetric Macdonald functions.

The goal of the present paper is to prove the conjecture of \cite{GO}. Further notations and the precise statement of the conjecture will be recalled along the way. In a few instances we will reprove known results of Ion--Wu (in the context of $\mathbb{B}_{t,q}$ rather than $\mathbb{A}_{t,q}$) and we will make clear when we do so. The work of Ion--Wu \cite{Ion_2022}, according to our conventions, deals with the $\mathbb{A}_{t,q}$ action on a larger space $W_{\bullet}$ (which we will not detail here) containing $V_{\bullet}$ as a submodule and relates this action to that of their stable-limit double affine Hecke algebra. When passing from $\mathbb{A}_{t,q}$ to $\mathbb{B}_{t,q}$ and $W_{\bullet}$ to $V_{\bullet}$ there are some subtleties and we reprove some of these known results in order to make these subtleties clear.

\subsection{Acknowledgements}
We would like to thank Nicolle Gonz\'{a}lez, Ben Goodberry, Eugene Gorsky, Mark Shimozono, and Monica Vazirani for helpful conversations and related collaborations. D.O. was supported by the Simons Foundation.

\section{Definitions and Conventions}

\subsection{Basic notation}

We work over the field $\K = \mathbb{Q}(q,t)$. 

For $n\in\Z_{\ge 0}$, we write $\Z_{\ge 0}^n=(\Z_{\ge 0})^n$. Elements of $\Z_{\ge 0}^n$ are called compositions. For $\nu\in\Z_{\ge 0}^n$, let $|\nu|=\nu_1+\dotsm+\nu_n$.

Let $\Y$ be the set of integer partitions. For $\la=(\la_1,\la_2,\dotsc)\in\Y$, let $\ell(\la)$ be its length, $|\la|=\la_1+\la_2+\dotsm$ its size, $n(\la)=\sum_{i\ge 1} (i-1)\la_i$, and $m_i(\la)$ the number of parts $\la_j$ of $\la$ equal to $i$, for all $i\ge 1$. Let $\Y_m$ be the set of partitions with length at most $m$. We identify $\Y_m$ with a subset of $\Z_{\ge 0}^m$ in the natural way. For any $\nu\in\Z_{\ge 0}^m$, we let $\mathrm{sort}(\nu)\in\Y_m$ be the weakly decreasing rearrangement of $\nu$.

Let $\mathfrak{S}_{n}$ be the symmetric group of permutations of the set $\{1,\ldots,n \}$ and let $s_1,\ldots,s_{n-1}$ denote the simple transpositions $s_i = (i, i+1).$ %Unless indicated otherwise, we regard $\mathfrak{S}_m\subset\mathfrak{S}_n$ as a subgroup in the natural way, for all $m\le n$. 

The length $\ell(w)$ of a permutation $w\in\fS_n$ is the minimal number of $s_i$ occurring in an expression $w = s_{i_1}\cdots s_{i_r}$, and when $r=\ell(w)$ such an expression is called reduced. We will write $w_{0}^{(n)}$ for the longest element of $\mathfrak{S}_{n}$ and when context is clear we will simply write $w_0 = w_0^{(n)}.$

For $\lambda \in \Y_n$, we define
    $$S_{\lambda}^{(n)}(t):= \sum_{w\in (\mathfrak{S}_n)_{\lambda}} t^{\ell(w)}\in\K$$
where $(\mathfrak{S}_n)_{\lambda}$ is the stabilizer of $\lambda$ in $\fS_n$. For $\la\in\Y$, we set $S_\la(t)=S_\la^{(\ell(\la))}(t)$. In terms of the $t$-factorials $[m]_{t}!:= \prod_{1 \leq i \leq m} \frac{1-t^i}{1-t}$, we have $S_\la(t)=\prod_{i\ge 1} [m_i(\la)]_t!$ and $S_\la^{(n)}(t)=[n-\ell(\la)]_t!\prod_{i\ge 1} [m_i(\la)]_t!$.

Let $\mathbbm{1}$ be the indicator function given by $\mathbbm{1}(\mathrm{True})=1$ and $\mathbbm{1}(\mathrm{False})=0$.

\subsection{Demazure-Lusztig operators}

Due to the differing conventions for the Demazure-Lusztig operators appearing across the literature, we will need to consider two versions of these operators $T_i, \mathsf{T}_{i}.$
\begin{defn}\label{d:dl}
For $1 \leq i \leq n-1$, define the \textbf{\textit{Demazure-Lusztig operators}}
$$T_i,\TT_i:\K[x_1,\ldots,x_n]\rightarrow \K[x_1,\ldots,x_n]$$
by 
\begin{align*}
T_i f &:= \frac{(t-1)x_{i+1}f + (x_{i+1}-tx_i)s_i(f)}{x_{i+1}-x_i}\\
\TT_i f &:= \frac{(t-1)x_if + (x_{i+1}-tx_i)s_i(f)}{x_{i+1}-x_i}.
\end{align*}
\end{defn}

The operators $T_i$ satisfy the relations:
\begin{enumerate}
    \item $(T_i-t)(T_i+1) = 0$
    \item $T_iT_j = T_jT_i$ if $|i-j| > 1$
    \item $T_iT_{i+1}T_i = T_{i+1}T_iT_{i+1}.$
\end{enumerate}

The operators $\mathsf{T}_i$ satisfy the same relations as the $T_i$ except for the quadratic relation. In this case we have $(\mathsf{T}_i -1)(\mathsf{T}_i + t) = 0.$

\begin{rem}
We may directly relate the operators $T_i$ and $\mathsf{T}_{i}$ by the following:
    $$T_i = t \mathsf{T}_i^{-1}.$$ 
\end{rem}

Whenever $w = s_{i_1}\cdots s_{i_r}$ is a reduced expression, we (unambiguously) define $T_{w} := T_{i_1}\cdots T_{i_r}$ and $\mathsf{T}_{w} := \mathsf{T}_{i_1}\cdots \mathsf{T}_{i_r}$.

\begin{defn}
    For $0\le m\le n$, we define the \textbf{\textit{partial symmetrizers}}
    $$P_{m}^{+}:= \sum_{ w\in \mathfrak{S}_m}T_{w}$$
    and %the \textbf{\textit{partially trivial idempotent}}
    $$\epsilon^{(n)}_{n-m} := \frac{1}{[m]_{t}!} \sum_{w \in \mathfrak{S}_{(1^{n-m},m)}} t^{{m \choose 2} - \ell(w)} \mathsf{T}_{w}$$ 
    where $\mathfrak{S}_m=\langle s_1,\dotsc,s_{m-1}\rangle\subset\mathfrak{S}_n$ and
    $\mathfrak{S}_{(1^{n-m},m)}=\langle s_{n-m+1},\ldots,s_{n-1}\rangle\subset\mathfrak{S}_{n}$.
\end{defn}

\subsection{Diagrams}
The diagram of a composition $\nu\in\Z_{\ge 0}^n$ is the subset $\dg(\nu)\subset \Z^2$ given by
$$ \dg(\nu)=\{(i,j) \mid 1\le i\le n, 1\le j\le\nu_i\}. $$
Elements of $\dg(\nu)$ are called boxes of $\nu$. We identify a composition with its diagram and write $(i,j)\in\nu$ to mean $(i,j)\in \dg(\nu)$. 
%We view the parts $\nu_i$ of $\nu$ as columns in $\dg(\nu)$.

For $\square=(i,j)\in\nu$, we will use the leg and arm functions
\begin{align*}
\ell_\nu(\square) &= \nu_i -j,\\
a_\nu(\square) &= \# \{ 1 \leq r < i \mid j\leq \nu_r \leq \nu_i\} + \# \{ i<r \leq n \mid j-1\leq \nu_r < \nu_i\}\\
\ta_\nu(\square) &= \# \{1\leq r < i \mid j \leq \nu_r \leq \nu_i \} + \# \{i<r \leq n \mid j \leq \nu_r < \nu_i \}.
\end{align*}

\begin{defn}[{\cite{goodberryarxiv}, \cite{lapointe2022msymmetric}}]\label{def:j}
For $\lambda \in \Y_m$ and $\gamma \in \mathbb{Z}_{\geq 0}^{k}$, define
    $$j_{(\lambda|\gamma)}(q,t) = \prod_{\square \in \lambda^{-}}(1-q^{\ell_\nu(\square)}t^{\tilde{a}_\nu(\square)+1}) \prod_{\square \in \gamma}(1-q^{\ell_\nu(\square)+1}t^{a_\nu(\square)+1})\in\K$$ 
    where $\nu=(\la_-,\gamma)$ and $\lambda_-$ is the weakly increasing rearrangement of $\lambda$.
    %given by $\lambda^{-}_{i}:= \lambda_{m-i+1}.$ 
    We view $\dg(\la_-)$ and $\dg(\ga)$ as subsets of $\dg(\nu)$ in the natural way.
    %all arms and legs are taken in the augmented diagram $\widehat{\mathrm{dg}}(\lambda^{-}, \gamma)$ and     
\end{defn}

\begin{example}
Let $\la=(2,1)$ and $\ga=(1,0)$. Then $\nu=(1,2, 1,0)$ and $\dg(\nu)$ is
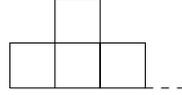
\begin{center}
\begin{tikzpicture}[scale=.6]
\draw (0,0) -- (0,1) -- (2,1) -- (2,0) -- (0,0);
\draw (1,0) -- (1,2) -- (2,2) -- (2,0);
\draw (2,0) -- (2,1) -- (3,1) -- (3,0) -- (2,0);
\draw[dashed] (3,0) -- (4,0);
\end{tikzpicture}
\end{center}
where the dashed segment indicates a part of size $0$. The values of $a_\nu(\square)$ or $\ta_\nu(\square)$---depending on whether $\square\in\la_-$ or $\square\in\ga$---are as follows:
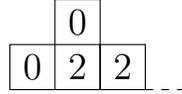
\begin{center}
\begin{tikzpicture}[scale=.6]
\draw (0,0) -- (0,1) -- (2,1) -- (2,0) -- (0,0);
\draw (1,0) -- (1,2) -- (2,2) -- (2,0);
\draw (2,0) -- (2,1) -- (3,1) -- (3,0) -- (2,0);
\draw[dashed] (3,0) -- (4,0);
\node at (.5,.5) {\large $0$};
\node at (1.5,.5) {\large $2$};
\node at (1.5,1.5) {\large $0$};
\node at (2.5,.5) {\large $2$};
\end{tikzpicture}
\end{center}
Using Definition~\ref{def:j}, we have
$$ j_{(2,1|1,0)}(q,t)=(1-q^{0}t^{1})(1-q^{1}t^{3})(1-q^{0}t^{1})(1-q^{1}t^{3}). $$
\end{example}

\subsection{Macdonald polynomials}

We review here the relevant conventions and constructions regarding Macdonald polynomials we will use in this paper.

\begin{defn}
\begin{enumerate}
\item For $\nu \in \mathbb{Z}_{\geq 0}^n$, denote by $E_{\nu} = E_{\nu}(x_1,\ldots,x_n;q,t) \in \K[x_1,\ldots,x_n]$ the nonsymmetric Macdonald polynomial of type $\mathrm{GL}_{n}$ as defined in \cite{haglund2007combinatorial}. 
\item For $\lambda \in \Y_m$ and $\gamma \in \mathbb{Z}_{\geq 0}^{k}$, let $P_{(\la|\ga)}\in\K[x_1,\ldots,x_n]^{\mathfrak{S}_m}$ be the partially symmetric Macdonald polynomial as defined in \cite{goodberryarxiv}:
$$P_{(\lambda|\gamma)}:= \frac{P_{m}^{+}E_{(\lambda|\gamma)}}{S_{\lambda}^{(m)}(t)}.$$
where $(\la|\ga)=(\la,\ga)\in\Z_{\ge 0}^n$ and $n=m+k$. Further, let $J_{(\la|\ga)}$ be the integral form partially symmetric Macdonald polynomial as defined in \cite{goodberryarxiv}:
$$J_{(\lambda|\gamma)}:= j_{(\lambda|\gamma)}P_{(\lambda|\gamma)}$$
where $j_{(\la|\ga)}$ is the scalar of Definition~\ref{def:j}.
\end{enumerate}
\end{defn}

Let $\Lambda$ be the ring of symmetric functions over $\K$ in the variables $x_1,x_2,\ldots.$ For $n = m+k$, we make the identification 
$$\K[x_1,\ldots,x_n]^{\mathfrak{S}_m} \cong \K[x_1,\ldots,x_m]^{\mathfrak{S}_m} \otimes \K[y_1,\ldots,y_k]$$ where 
$x_i \mapsto x_i$ for $1 \leq i \leq m$ and $x_{m+i} \mapsto y_i$ for $1\leq i \leq k.$ Taking the graded inverse limit as $m\to\infty$, we arrive at
$$V_k:= \Lambda \otimes \K[y_1,\ldots,y_k].$$
We write elements of $V_k$ as $f(X|y)$, where $X=x_1+x_2+\dotsm$ in plethystic notation and $y=(y_1,\dotsc,y_k)$. We will sometimes write $V_{\bullet} := \bigoplus_{k \geq 0} V_k.$ 

\begin{defn}[]
     For $\la\in\Y$ and $\gamma \in \mathbb{Z}_{\geq 0}^k$, define 
    $$P_{(\lambda|\gamma)}\left(X | y\right):= \lim_{m \rightarrow \infty} P_{(\lambda|\gamma)}(x_1,\ldots,x_m,y_1,\ldots,y_{k}) \in V_k$$ and similarly 
    $$J_{(\lambda|\gamma)}\left(X | y\right):= \lim_{m \rightarrow \infty} J_{(\lambda|\gamma)}(x_1,\ldots,x_m,y_1,\ldots,y_{k}) \in V_k.$$ 
\end{defn}

Here $\lim_{m \rightarrow \infty}$ denotes the classical stable-limit (as opposed to a more subtle notion of limit discussed below). This definition is from \cite{GO}, but it is based on the the following stability result proved in \cite{Goodberry} and \cite{lapointe2022msymmetric}:
$$P_{(\lambda,0|\gamma)}(x_1,\ldots,x_m,0,y_1,\ldots,y_{k}) =P_{(\lambda|\gamma)}(x_1,\ldots,x_m,y_1,\ldots,y_{k}).$$ 

%It is shown in \cite{GO} that we also have the analogous result for $J_{(\lambda|\gamma)}.$

\subsection{Limit Cherednik operators}

We will need to discuss a related construction of Ion--Wu \cite{Ion_2022} and a few results from \cite{MBWArxiv}. We will see later in Corollary \ref{cor:tEJ} precisely how these two constructions are related.

%\begin{defn}[\cite{Ion_2022}]
    %Define the space of \textbf{\textit{almost symmetric functions}} 
    %$$\mathscr{P}_{as}^{+}:= \lim_{\rightarrow}\K[x_1,\ldots, x_k] \otimes \Lambda(x_{k+1}+\ldots) .$$
%\end{defn}

\begin{defn}[\cite{Ion_2022}]
Define the space of \textbf{\textit{almost symmetric functions}} 
    $$\mathscr{P}_{as}^{+}:= \lim_{\rightarrow}\K[x_1,\ldots, x_k] \otimes \Lambda(x_{k+1}+\ldots) $$
where $\Lambda(x_{k+1}+\cdots)$ is the $\K$-algebra of symmetric functions in the variables $x_{k+1},\dotsc$.
Define operators $\mathscr{Y}_i$ for $i \geq 1$ on $\mathscr{P}_{as}^{+}$ as
$$\mathscr{Y}_1:= \underset{m}{\widetilde{\lim}}\, t^m \rho Y_{1}^{(m)}$$ 
where $\underset{m}{\widetilde{\lim}}$ denotes the $t$-adic stable-limit as defined Ion--Wu \cite{Ion_2022} and
$$\mathscr{Y}_{i+1}:= t^{-1}\mathsf{T}_i \mathscr{Y}_i \mathsf{T}_i$$ for $i \geq 1$, 
    where 
    \begin{enumerate}
        \item $Y_{1}^{(m)}:= \omega_{m}^{-1}\mathsf{T}_{m-1}^{-1}\cdots \mathsf{T}_{1}^{-1}$, where $\omega_{m}^{-1}(x_1^{a_1}\cdots x_m^{a_m}) = x_2^{a_1}\cdots x_{m}^{a_{m-1}}(qx_1)^{a_{m}}$
        \item $\rho(x_1^{a_1}\cdots x_m^{a_m}) = \mathbbm{1}(a_1 > 0)x_1^{a_1}\cdots x_m^{a_m}$.
    \end{enumerate}
\end{defn}

The $t$-adic stable-limit of Ion--Wu \cite{Ion_2022} is a more general (weaker) notion of convergence than the classical stable-limit. The distinction between these two definitions will not be important for any of the arguments in this paper.

\begin{thm}[\cite{Ion_2022}]
   The operators $\mathscr{Y}_i$ for $ i \geq 1$ are well-defined and mutually commute. 
\end{thm}

The action of the operators $\mathscr{Y}_i$ on $\mathscr{P}_{as}^+$ are simultaneously diagonalizable with a basis generalizing the finite-rank nonsymmetric Macdonald polynomials. We say a composition $\mu = (\mu_1,\ldots,\mu_k)$ is \textbf{\textit{reduced}} if $\mu_k \neq 0$ and we call $\ell(\mu) = k$ the length of $\mu.$ Given two vectors $v\in \mathbb{Z}^n$, $w \in \mathbb{Z}^m$ we define $v*w \in \mathbb{Z}^{n+m}$ as the concatenation of $v$ and $w.$ If $\lambda\in\Y$, we regard $\la$ as an element of $\Z_{\ge 0}^{\ell(\la)}$ when using this notation.

\begin{defn}[\cite{MBWArxiv}] \label{stable-limit non-sym MacD function defn}
    For a reduced composition $\mu$ and a partition $\lambda$ define the \textbf{\textit{stable-limit nonsymmetric Macdonald function}} $\widetilde{E}_{(\mu|\lambda)}\in\mathscr{P}_{as}^+$ as 
    $$\widetilde{E}_{(\mu|\lambda)}(x_1,x_2,\ldots;q,t) := \underset{n}{\widetilde{\lim}}\,  \epsilon_{\ell(\mu)}^{(n)}(E_{\mu*\la*0^{n-(\ell(\mu)+\ell(\la))}}(x_1,\ldots,x_n;q^{-1},t)). $$
\end{defn}

We will be using the following result later in this paper.

\begin{thm}[\cite{MBWArxiv}]\label{weight basis theorem}
    The $\widetilde{E}_{(\mu|\lambda)}$ are a $\mathscr{Y}$-weight basis for $\mathscr{P}_{as}^{+}$. Further, we have that 
    $$\mathscr{Y}_i(\widetilde{E}_{(\mu|\lambda)}) = \widetilde{\alpha}_{(\mu|\lambda)}(i) \widetilde{E}_{(\mu|\lambda)} $$ where $\widetilde{\alpha}_{(\mu|\lambda)}(i)$ is given explicitly by 

    \[ \widetilde{\alpha}_{(\mu|\lambda)}(i) = 
    \begin{cases}
    q^{\mu_i}t^{\ell(\mu)+\ell(\lambda)+1-\beta_{\mu*\lambda}(i)} & \text{if $i\leq \ell(\mu)$ and $\mu_i \neq 0$} \\
    0 & \text{otherwise}
     \end{cases}
\]
and 
$$\beta_{\nu}(i) := \#\{j: 1\leq j \leq i,\, \nu_j \leq \nu_i\} + \#\{j: i < j \leq \ell(\nu),\, \nu_i > \nu_j\}.$$
\end{thm}

\subsection{The algebra $\B_{t,q}$ and its geometric representation}

Let $\A_{t,q}$ be the Carlsson-Mellit algebra, which is a $\Q(q,t)$-algebra generated by orthogonal idempotents $\id_0,\id_1,\dotsc$ and elements $$\dd_-,\dd_+,\dd_+^*,\TT_1,\TT_2,\dotsc, \yy_1,\yy_2,\dotsc,\zz_1,\zz_2,\dotsc$$ satisfying certain relations (see \cite{GCM_2017}, noting that we swap the roles $q$ and $t$). Let $\B_{t,q}\subset\A_{t,q}$ be the subalgebra generated by all generators of $\A_{t,q}$ except $\dd_+^*$.

For $n\in\Z_{\ge 0}$ and $0\le k\le n$, the parabolic flag Hilbert scheme $\PFH_{n,n-k}$ is a smooth variety parametrizing flags of ideals
$$ I_n\subset I_{n-1}\subset\dotsm\subset I_{n-k} \subset \C[x,y] $$
such that $\dim \C[x,y]/I_s=s$ for all $s$ and $yI_{n-k}\subset I_n$. The space $\PFH_{n,n-k}$ carries a natural action of $\T=(\C^\times)^2$ with isolated fixed points given by chains as above in which each $I_s$ is a \textit{monomial} ideal. These are naturally indexed as $I_{\mu,w}$ by pairs $(\mu,w)$, where $\mu$ is an integer partition of size $n$ and $w=(w_1,\dotsc,w_k)$ is a list $\T$-weights corresponding to a horizontal strip of size $k$ in $\mu$ such that the removal of boxes corresponding to $w_1,\dotsc,w_k$ in this order always leaves a partition. We refer to \cite{GCM_2017} for more details, and to \cite{GO} for our specific conventions.

We identify $\Z[q^{\pm 1},t^{\pm 1}]$ with $K_\T(\mathrm{pt})$ (the representation ring of $\T$), and $\K$ with its field of fractions. The localized equivariant $K$-groups $K_\T(\PFH_{n,n-k})_\mathrm{loc}$ are a finite-dimensional $\K$-vector spaces, with bases given by the classes $[I_{\mu,w}]$ of skyscraper sheaves over the fixed points $I_{\mu,w}$ described above (and depending on $n$ and $k$). We will write $U_k:= \bigoplus_{n \geq k} K_\T(\PFH_{n,n-k})_\mathrm{loc}$ for $k \geq 0$. The spaces $U_\bullet=(U_k)_{k\ge 0}$ affords the geometric representation of $\A_{t,q}$ \cite{GCM_2017}. The restriction of this action to $\B_{t,q}$ is given on the fixed-point bases as follows:
\begin{align}
\label{dd-geom}
\dd_-\cdot [I_{\mu,wx}] &=[I_{\mu,w}]\\
\label{TT-geom}
\TT_i\cdot [I_{\mu,w}] &= \frac{(t-1)w_{i+1}}{w_i-w_{i+1}}[I_{\mu,w}]+\frac{w_i-tw_{i+1}}{w_i-w_{i+1}}[I_{\mu,s_i(w)}]\\
\label{zz-geom}
\zz_i\cdot [I_{\mu,w}] &= w_i [I_{\mu,w}]
\end{align}
where $I_{\mu,w}\in\PFH_{n,n-k}$, $\TT_i$ has $1\le i<k$ and $\zz_i$ has $1\le i\le k$. Further,
%Further, $\id_k$ is the natural projection onto $\bigoplus_{n\ge k}K_\T(\PFH_{n,n-k})_\mathrm{loc}$, we have $\dd_-|_{K_\T(\PFH_{n,n})_\mathrm{loc}}=0$, and
\begin{align}
\dd_+ \cdot[I_{\mu,w}] &= -t^k \sum_{(\nu,xw)} xd_{\nu,\mu}\prod_{i=1}^k \frac{x-qw_i}{x-qtw_i}[I_{\nu,xw}]
\end{align}
where the summation runs over all indices $(\nu,xw)$ of $\T$-fixed points in $\PFH_{n+1,n-k}$ such that $\nu/\mu$ is a single box skew shape, and $d_{\nu,\mu}$ is a Macdonald-Pieri coefficient (see \cite{GCM_2017}). We note that the $\yy_i$ are expressed in terms of $\dd_+,\dd_-$ and the $\TT_i$. In this paper, we will only need the formulas \eqref{dd-geom}, \eqref{TT-geom}, and \eqref{zz-geom}.

\subsection{Polynomial representation}

The spaces $V_{\bullet} =(V_k)_{k\ge 0}$ afford the polynomial representation of $\B_{t,q}$ \cite{GCM_2017}. On $V_k=\Lambda\otimes\K[y_1,\dotsc,y_k]$, $\TT_1,\dotsc,\TT_{k-1}$ are the Demazure-Lusztig operators of Definition~\ref{d:dl} in the variables $y_1,\dotsc,y_k$, and $\yy_1,\dotsc,\yy_k$ are the corresponding multiplication operators by these variables. The operators $\dd_- : V_{k+1}\to V_k$ and $\dd_+: V_k \to V_{k+1}$ are given by
\begin{align*}
\dd_-\cdot f &=-f(X-(t-1)y_{k+1})\Omega(-y_{k+1}^{-1}X)\Big|_{y_{k+1}^{-1}}\\ 
\dd_+\cdot F &= \TT_1\dotsm \TT_k\cdot f(X+(t-1)y_{k+1}),
\end{align*}
where $A(y)|_{y^s}=A_s$ for a formal series $\sum_{s\in\Z} A_s y^s$, and $\Omega(-yX)=\sum_{i\ge 0} (-y)^i e_i(X)$ with $e_i$ denoting elementary symmetric functions.

We recall the necessary further details of this action and related results of \cite{Ion_2022} as they are needed below. Furthermore, as shown in \cite{GCM_2017}, there is a unique $\mathbb{B}_{t,q}$-module isomorphism $\Phi_\bullet:U_\bullet\overset{\sim}{\to} V_\bullet$ such that $\Phi_0([I_{\varnothing,\varnothing}])=1\in V_0$.

\section{Conjecture and Proof Strategy}

Here we review the main conjecture from \cite{GO} and outline the proof of its verification (Theorem \ref{main thm}) in this paper. For $F(X|y;q,t) \in V_k$ we write $F(X|y)^{*}:= F(X|y;q,t^{-1}).$  

\begin{defn}[Goodbery-Orr \cite{GO}]\label{modified partially sym MacD function}
    For $\lambda \in \mathbb{Y}$ and $\gamma \in \mathbb{Z}_{\geq 0}^k$ define the \textbf{\textit{modified partially symmetric Macdonald function}}
    $$\widetilde{H}_{(\lambda|\gamma)}(X|y):= t^{n(\mathrm{sort}(\lambda,\gamma))+|(\lambda|\gamma)|}\left(t^{-\ell(w_0^{(k)})}w_0^{(k)} T_{w_0^{(k)}}\mathcal{J}_{(\lambda|\gamma)}\left( \frac{X}{t^{-1}-1}\Big| y \right) \right)^{*}\in V_k$$
    where $w_0^{(k)}$ and $T_{w_0^{(k)}}$ are acting on the variables $y_1,\dotsc,y_k$.
\end{defn}

As in \cite{GCM_2017}, for a $\mathbb{T}$-fixed point $I_{\mu,w}$ let 
$$H_{\mu,w}:= (-1)^{|\mu|}t^{n(\mu')}q^{n(\mu)}[I_{\mu,w}].$$ We write $\phi$ for the bijection $\phi(\mu,w) = (\lambda,\gamma)$ between the indexing sets of the $\mathbb{T}$-fixed points $I_{\mu,w}$ of the $\PFH_{n,n-k}$ and pairs $(\lambda|\gamma)$ with $\lambda \in \mathbb{Y}$ and $\gamma \in \mathbb{Z}_{\geq 0}^k$ defined in \cite{GCM_2017} and detailed in \cite{GO}. In particular, if $\phi(\mu,w) = (\lambda,\gamma)$, then $\mu = \sort(\gamma_1+1,\ldots,\gamma_k+1,\lambda).$

\begin{conj}[\cite{GO}]\label{main conj}
    
The $\mathbb{B}_{t,q}$-module isomorphism $\Phi_\bullet : U_\bullet \to V_\bullet$ 
%\begin{align}
%\bigoplus_{k\ge 0}\bigoplus_{n \ge k} K_\T(\PFH_{n,n-k})_\mathrm{loc} \overset{\Phi}{\longrightarrow} \bigoplus_{k\ge 0}\Lambda\otimes\K[y_1,\dotsc,y_k]
%\end{align}
maps each normalized fixed point class $H_{\mu,w}$ to $\widetilde{H}_{(\lambda|\gamma)}(X|y)$ for $\phi(\mu,w) = (\lambda|\gamma).$
\end{conj}

In this paper, we confirm this conjecture:

\begin{thm-main}[Theorem~\ref{main thm}]
    Whenever $\phi(\mu,w) = (\lambda|\gamma)$, 
    $\Phi(H_{\mu,w})= \widetilde{H}_{(\lambda|\gamma)}.$
\end{thm-main}

The strategy for our proof of Theorem \ref{main thm} proceeds as follows:
\begin{itemize}
\item First, in Proposition \ref{prop:tEP} we bridge the two perspectives on partially symmetrized Macdonald polynomials showing that every partially symmetric Macdonald function $P_{(\lambda|\gamma)}$ of Goodberry \cite{goodberryarxiv} and Lapointe \cite{lapointe2022msymmetric} may be obtained via a transformation from a corresponding stable-limit nonsymmetric Macdonald function $\widetilde{E}_{(\mu|\nu)}.$ Then, using a remarkable identity \eqref{Concha-Lapointe identity} of Concha-Lapointe \cite{Concha_Lapointe}, we show in Proposition \ref{Calculations of eigenvectors} that applying Ion--Wu's isomorphism $\Xi$ to each $\widetilde{E}_{(\gamma_1+1,\ldots, \gamma_k+1|\lambda)}$ yields exactly $\widetilde{H}_{(\lambda|\gamma)}$ up to explicit non-zero scalars. 
\item Next, in Theorem \ref{prop: limit cherednik action explicit} we use the spectral theory of the limit Cherednik operators $\mathscr{Y}_i$ and an operator $\Delta_{p_1}$ to verify that each $\Phi(H_{\mu,w})$ must agree with $\widetilde{H}_{(\lambda|\gamma)}$ up to a non-zero scalar $\beta_{\mu,w}$. In particular, we use the fundamental fact that the polynomial representation of $\mathbb{B}_{t,q}^{\text{ext}}$ has simple joint $\Delta_{p_1},\zz_i$ spectrum \cite{gonzález2023calibrated} and the corresponding statement about the stable-limit DAHA polynomial representation \cite{MBWArxiv}. 
\item In order to show that these scalars $\beta_{\mu,w} = 1$, we verify in Lemma \ref{d- lemma} that the action of $\dd_{-}$ on both the geometric side, $H_{\mu,w}$, and the algebraic side, $\widetilde{H}_{(\lambda|\gamma)}$, agree. From here in Proposition \ref{scalars are all 1} we use as a base case for an induction argument the fact that, due to Carlsson--Gorsky--Mellit \cite{GCM_2017}, $\Phi(H_{\mu,\emptyset}) = \widetilde{H}_{\mu}$. This argument explicitly requires the exact normalization of $\widetilde{H}_{(\lambda|\gamma)}$ from \cite{GO} and the fact from \cite{GO} that the action of the Demazure-Lusztig operators $\mathsf{T}_i$ on both the $H_{\mu,w}$ and the $\widetilde{H}_{(\lambda|\gamma)}$ align up to the bijection $\phi(\mu,w) = (\lambda|\gamma)$. This directly implies Theorem \ref{main thm} verifying Conjecture \ref{main conj}.
\end{itemize}

\begin{example}
We have $\phi((2,1),(q))=(1|1)$ and
$$
\widetilde{H}_{(1|1)} = (t^2-q)e_1y_1+h_2+te_2\in\Lambda\otimes\K[y_1].
$$
where $e_i,h_i$ are the elementary and complete symmetric functions. One may directly compute that
$$
\dd_-\cdot\widetilde{H}_{(1|1)} = s_3+(q+t)s_{(2,1)}+qts_{(1,1,1)}=\widetilde{H}_{(2,1)}
$$
where $s_\la$ are the Schur functions and $\widetilde{H}_{(2,1)}$ is the modified Macdonald function.
\end{example}

\section{Aligning Conventions}\label{align conventions section}

In this section, we will align the conventions between the partially symmetric Macdonald functions of \cite{goodberryarxiv} and the stable-limit nonsymmetric Macdonald polynomials of \cite{MBWArxiv}.

\begin{lem}\label{k-twist up lemma}
For any $\lambda\in(\Z_{\ge 0})^m$ and $\gamma\in(\Z_{\ge 0})^k$, we have
\begin{align*}
&E_{(\gamma_1+1,\ldots,\gamma_k+1,\lambda_1,\ldots, \lambda_m)}(x_1,\ldots, x_k, x_{k+1},\ldots x_{k+m}) \\
&\qquad\qquad = q^{\gamma_1+\cdots + \gamma_k} x_1\cdots x_k E_{(\lambda_1,\ldots, \lambda_m,\gamma_1,\ldots, \gamma_k)}(x_{k+1},\ldots, x_{k+m}, q^{-1}x_1,\ldots, q^{-1}x_k).
\end{align*}
\end{lem}
\begin{proof}
    Repeatedly apply the Knop-Sahi relation (in the conventions of \cite{haglund2007combinatorial})
    $$E_{(\mu_n+1,\mu_1,\ldots, \mu_{n-1})}(x_1,\ldots, x_n) 
    %= q^{\mu_n}x_1\pi_n(E_{\mu}(x_1,\ldots,x_n)) 
    = q^{\mu_n}x_1E_{\mu}(x_2,\ldots,x_{n}, q^{-1}x_1)$$
    which holds for any $\mu\in(\Z_{\ge 0})^n$.
    %Note that 
    %$(x_1\pi_n)^{k} = x_1\cdots x_k\pi_n^{k}$ so for any $f(x)$ we see
    %$$(x_1\pi_n)^{k}(f(x_1,\ldots,x_n)) = x_1\cdots x_k f(x_{k+1}, \ldots, x_{k+m},q^{-1}x_1,\ldots, q^{-1}x_k).$$ Applying this to $\mu = (\lambda_1,\ldots, \lambda_m,\gamma_1,\ldots,\gamma_k)$ with $n = m +k$ shows that 
    %\begin{align*}
    %&E_{(\gamma_1+1,\ldots,\gamma_k+1,\lambda_1,\ldots, \lambda_m)}(x_1,\ldots, x_k, x_{k+1},\ldots x_{k+m}) \\
    %&= q^{\gamma_1+\ldots + \gamma_k}(x_1\pi_n)^{k}(E_{(\lambda_1,\ldots, \lambda_m,\gamma_1,\ldots, \gamma_k)}(x_1,\ldots, x_m,x_{m+1},\ldots, x_{m+k})) \\
    %&= q^{\gamma_1+\ldots + \gamma_k}x_1\cdots x_kE_{(\lambda_1,\ldots, \lambda_m,\gamma_1,\ldots, \gamma_k)}(x_{k+1}, \ldots, x_{k+m},q^{-1}x_1,\ldots, q^{-1}x_k).\qedhere
    %\end{align*}
\end{proof}

\begin{prop}\label{prop:tEP}
For any $\lambda\in\Y$ and $\gamma\in(\Z_{\ge 0})^k$,
    \begin{align*}
    &\prod_{i}\prod_{j=1}^{m_i(\lambda)}( 1-t^{j})^{-1} \widetilde{E}_{(\gamma_1+1,\ldots, \gamma_k+1|\lambda)}(qx_1,\ldots,qx_k,x_{k+1},x_{k+2},\ldots;q^{-1},t) \\
    &\qquad\qquad =q^{\gamma_1+\cdots + \gamma_k}x_1\cdots x_k P_{(\lambda|\gamma)}(x_{k+1}+\cdots| x_1,\ldots,x_k).
    \end{align*}
\end{prop}
\begin{proof}
    Choose any $m$ such that $\la\in\Y_m$. Starting from Lemma \ref{k-twist up lemma}, we apply the normalized symmetrizer $\frac{P_{[k+1,\ldots,k+m]}^{+}}{S_{\lambda}^{(m)}(t)}$ to both sides and then take the limit as $m\to\infty$.
    % to obtain:
    %\begin{align*}
    %    &\frac{P_{[k+1,\ldots,k+m]}^{+}}{S_{\lambda*0^{m- \ell(\lambda)}}(t)}E_{(\gamma_1+1,\ldots,\gamma_k+1,\lambda_1,\ldots, \lambda_m)}(x_1,\ldots, x_k, x_{k+1},\ldots x_{k+m}) \\
    %    &= \frac{P_{[k+1,\ldots,k+m]}^{+}}{S_{\lambda*0^{m- \ell(\lambda)}}(t)}(q^{\gamma_1+\ldots + \gamma_k} x_1\cdots x_k E_{(\lambda_1,\ldots, \lambda_m,\gamma_1,\ldots, \gamma_k)}(x_{k+1},\ldots, x_{k+m}, q^{-1}x_1,\ldots, q^{-1}x_k)).\\
    %\end{align*}
    
    %Now we limit $m \rightarrow \infty$ on both sides. On one hand, 
    On the left-hand side, we obtain
    \begin{align*}
        &\frac{P_{[k+1,\ldots,k+m]}^{+}}{S_{\lambda}^{(m)}(t)}E_{(\gamma_1+1,\ldots,\gamma_k+1,\lambda_1,\ldots, \lambda_m)}(x_1,\ldots, x_k, x_{k+1},\ldots x_{k+m}) \\
        &\qquad\qquad = \frac{[m]_{t}!}{S_{\lambda}(t) [m-\ell(\lambda)]_{t}!} \epsilon^{(m+k)}_{k}(E_{(\gamma_1+1,\ldots,\gamma_k+1,\lambda_1,\ldots, \lambda_m)}(x_1,\ldots, x_k, x_{k+1},\ldots x_{k+m})),
    \end{align*}
    which limits to 
    \begin{align*} 
    &\frac{1}{(1-t)^{\ell(\lambda)}S_{\lambda}(t)} \widetilde{E}_{(\gamma_1+1,\ldots, \gamma_k+1|\lambda)}(x_1,\ldots,x_k,x_{k+1},x_{k+2},\ldots;q^{-1},t) \\
    &\qquad\qquad = \prod_{i\ge 1}\prod_{j=1}^{m_i(\lambda)}( 1-t^{j})^{-1} \widetilde{E}_{(\gamma_1+1\ldots, \gamma_k+1|\lambda)}(x_1,\ldots,x_k,x_{k+1},x_{k+2},\ldots;q^{-1},t).
    \end{align*}

On the other hand,
    \begin{align*}
       &\frac{P_{[k+1,\ldots,k+m]}^{+}}{S_{\lambda}^{(m)}(t)}(q^{\gamma_1+\ldots + \gamma_k} x_1\cdots x_k E_{(\lambda_1,\ldots, \lambda_m,\gamma_1,\ldots, \gamma_k)}(x_{k+1},\ldots, x_{k+m}, q^{-1}x_1,\ldots, q^{-1}x_k)) \\
       &\qquad\qquad = q^{\gamma_1+\ldots + \gamma_k} x_1\cdots x_k P_{(\lambda|\gamma)}(x_{k+1},\ldots, x_{k+m}|q^{-1}x_1,\ldots, q^{-1}x_k),
    \end{align*}
    which limits to 
    $q^{\gamma_1+\ldots + \gamma_k} x_1\cdots x_kP_{(\lambda|\gamma)}(x_{k+1}+\ldots|q^{-1}x_1,\ldots, q^{-1}x_k).$

    By substituting $x_i \mapsto qx_i$ for $1 \leq i \leq k$, we obtain the result.
\end{proof}

\begin{defn}
    For $k \geq 0$, let $\mathfrak{X}_k := x_{k+1}+x_{k+2}+\ldots.$ Define the map 
    $$ \Xi_k :  \K[x_1,\ldots,x_k]^+\otimes \Lambda[\mathfrak{X}_k] \longrightarrow V_k,$$ 
    where $\K[x_1,\ldots,x_k]^+=x_1\cdots x_k \K[x_1,\ldots,x_k]$, by the formula
    $$\Xi_{k}(x_1^{a_1+1}\cdots x_{k}^{a_k+1}F[\mathfrak{X}_k]) = y_1^{a_1}\cdots y_k^{a_k}F\left(\frac{X}{t-1}\right).$$
\end{defn}

\begin{cor}\label{cor:tEJ}
For any $\la\in\Y$ and $\gamma\in(\Z_{\ge 0})^k$, we have
    $$\mathcal{J}_{(\lambda|\gamma)}\left(\frac{X}{t-1}\Big| y_1,\ldots,y_k \right) = \frac{j_{(\lambda|\gamma)}q^{-(\gamma_1+\cdots + \gamma_k)}}{\prod_{i\ge 1}\prod_{j=1}^{m_i(\lambda)}( 1-t^{j})}\, \sh\circ \Xi_k \left( \widetilde{E}_{(\gamma_1+1,\ldots \gamma_k+1|\lambda)}(x_1,\ldots ;q^{-1},t) \right)%\left(X| qy_1,\ldots, qy_k \right)
    $$
where $\sh : V_k \to V_k$ is the map
$$ \mathrm{sh}\big(f(X\mid y)\big)=f(X\mid qy_1,\dotsc,qy_k). $$
\end{cor}
\begin{proof}
We compute using Proposition~\ref{prop:tEP} and the definition of $\Xi_k$ that
    \begin{align*}
        &\mathcal{J}_{(\lambda|\gamma)}\left(\frac{X}{t-1}\Big| y\right)\\
        %&= \frac{y_1\cdots y_k\mathcal{J}_{(\lambda|\gamma)}\left(\frac{X}{t-1}\Big| y\right) }{y_1\cdots y_k}\\
        &= \Xi_k \left( x_1\cdots x_k \mathcal{J}_{(\lambda|\gamma)}\left(x_{k+1}+\cdots\Big| x_1,\ldots,x_k\right) \right) \\
        &= \Xi_k \left( j_{(\lambda|\gamma)}x_1\cdots x_k P_{(\lambda|\gamma)}\left(x_{k+1}+\cdots\Big| x_1,\ldots,x_k\right) \right)\\
        &= \Xi_k \left( \frac{j_{(\lambda|\gamma)}q^{-(\gamma_1+\cdots+\gamma_k )}}{\prod_{i\ge 1}\prod_{j=1}^{m_i(\lambda)}( 1-t^{j})}\widetilde{E}_{(\gamma_1+1,\ldots, \gamma_k+1|\lambda)}(qx_1,\ldots, qx_k, x_{k+1},\ldots;q^{-1},t) \right) \\
        &=  \frac{j_{(\lambda|\gamma)}q^{-(\gamma_1+\cdots+\gamma_k )}}{\prod_{i\ge 1}\prod_{j=1}^{m_i(\lambda)}( 1-t^{j})}\, \Xi_k \left(\widetilde{E}_{(\gamma_1+1,\ldots, \gamma_k+1|\lambda)}(qx_1,\ldots, qx_k, x_{k+1},\ldots;q^{-1},t) \right)\\
        &= \frac{j_{(\lambda|\gamma)}q^{-(\gamma_1+\cdots+\gamma_k )}}{\prod_{i\ge 1}\prod_{j=1}^{m_i(\lambda)}( 1-t^{j})
        }\,\sh\circ\Xi_k \left(\widetilde{E}_{(\gamma_1+1,\ldots, \gamma_k+1|\lambda)}(x_1,\ldots;q^{-1},t) \right). %\left(X\Big| qy_1,\ldots, qy_k \right)
        \qedhere
    \end{align*}
\end{proof}

The following identity of Concha--Lapointe is crucial in connecting \cite{GO} with \cite{MBWArxiv}.

\begin{thm}[{\cite{Concha_Lapointe}}]\label{Concha-Lapointe identity}
For any partition $\lambda$ and $\gamma\in(\Z_{\ge 0})^k$, we have
$$\mathrm{sh}\big(P_{(\lambda|\gamma)}(X\mid y)^\dag\big) = q^{\gamma_1+\dotsm+\gamma_k}t^{\mathrm{inv}(\gamma)-\ell(w_0^{(k)})}w_0^{(k)}T_{w_0^{(k)}}P_{(\lambda|\gamma)}(X\mid y)$$
where $\dag$ is the involution on coefficients sending $q$ and $t$ to their inverses, and
$$\mathrm{inv}(\gamma)=|\{i<j : \gamma_i>\gamma_j\}|.$$
%and
%$$ \mathrm{sh}\big(f(X\mid y)\big)=f(X\mid qy_1,\dotsc,qy_k). $$
\end{thm}

\begin{example}
For $(\la|\ga)=(\varnothing|1,0)$, we have
\begin{align*}
    P_{(\varnothing|1,0)} &= \frac{1-t}{1-qt^2}e_1(X) + y_1\\ 
    \mathrm{sh}\big((P_{(\varnothing|1,0)})^\dag\big) &= \frac{1-t^{-1}}{1-q^{-1}t^{-2}}e_1(X) + qy_1\\
    &= \frac{t^{-1}}{q^{-1}t^{-2}}\frac{1-t}{1-qt^{2}}e_1(X) + qy_1\\
    &= qt\frac{1-t}{1-qt^{2}}e_1(X) + qy_1\\
    q^1 t^{1-1}w_0^{(2)}T_{w_0^{(2)}}P_{(\varnothing|1,0)} &= qt\frac{1-t}{1-qt^2}e_1(X) + qy_1.   
\end{align*}
\end{example}

Now we complete our match between \cite{GO} and \cite{MBWArxiv}.

\begin{prop}\label{Calculations of eigenvectors}
For any $\lambda\in\Y$ and $\gamma\in(\Z_{\ge 0})^k$, we have
\begin{equation}\label{eq eigenvectors}
    \Xi_{k}(\widetilde{E}_{(\gamma_1+1,\ldots, \gamma_k+1|\lambda)}) = \frac{\prod_{i\ge 1}\prod_{j=1}^{m_i(\lambda)}( 1-t^{j})}{j_{(\lambda|\gamma)}(q,t^{-1})t^{\mathrm{inv}(\gamma)-\ell(w_0^{(k)})}} \left(w_0^{(k)} T_{w_0^{(k)}} \mathcal{J}_{(\lambda|\gamma)}\left( \frac{X}{t^{-1}-1} \Big| y_1,\ldots , y_k \right) \right)^{*}.
\end{equation}

\end{prop}
\begin{proof}
    %From the prior calculation we found that 
    %$$\frac{j_{(\lambda|\gamma)}q^{-(\gamma_1+\cdots + \gamma_k)}}{\prod_{i\ge 1}\prod_{j=1}^{m_i(\lambda)}( 1-t^{j})}\,\sh\circ\Xi_k \left( \widetilde{E}_{(\gamma_1+1,\ldots \gamma_k+1|\lambda)}(x_1,\ldots ;q^{-1},t) \right) = \mathcal{J}_{(\lambda|\gamma)}\left(\frac{X}{t-1}\Big| y_1,\ldots,y_k \right)$$
    %which gives
    %$$\Xi_k \left( \widetilde{E}_{(\gamma_1+1,\ldots \gamma_k+1|\lambda)}(x_1,\ldots ;q^{-1},t) \right)\left(X| qy_1,\ldots, qy_k \right) = j_{(\lambda|\gamma)}^{-1}q^{\gamma_1+\ldots+\gamma_k}\prod_{i}\prod_{j=1}^{m_i(\lambda)}( 1-t^{j})\mathcal{J}_{(\lambda|\gamma)}\left(\frac{X}{t-1}\Big| y_1,\ldots,y_k \right).$$
    %Now we scale the $y$ variables on both sides by $q^{-1}$ and note that this scaling commutes with $\Xi_{k}$ to show:
    Using the definition of $\mathcal{J}_{(\lambda|\gamma)}$, we can write Corollary~\ref{cor:tEJ} as
    $$\Xi_k \left( \widetilde{E}_{(\gamma_1+1,\ldots \gamma_k+1|\lambda)}(x_1,\ldots ;q^{-1},t) \right) = \frac{\prod_{i\ge 1}\prod_{j=1}^{m_i(\lambda)}( 1-t^{j})}{q^{-(\gamma_1+\cdots+\gamma_k)}}P_{(\lambda|\gamma)}\left(\frac{X}{t-1}\Big| q^{-1}y_1,\ldots,q^{-1}y_k \right).   $$
    By sending $q \mapsto q^{-1}$ we find
    \begin{align*}
    \Xi_k \left( \widetilde{E}_{(\gamma_1+1,\ldots \gamma_k+1|\lambda)} \right) &= \frac{\prod_{i\ge 1}\prod_{j=1}^{m_i(\lambda)}( 1-t^{j})}{q^{\gamma_1+\cdots+\gamma_k}}P_{(\lambda|\gamma)}\left(\frac{X}{t-1}\Big| qy_1,\ldots,qy_k;q^{-1},t \right)\\
    &= \frac{\prod_{i\ge 1}\prod_{j=1}^{m_i(\lambda)}( 1-t^{j})}{q^{\gamma_1+\cdots+\gamma_k}}\,\sh\left(P_{(\lambda|\gamma)}\left(\frac{X}{t^{-1}-1}\Big| y \right)^\dag\right)^*.
    \end{align*}
    Finally, the Concha--Lapointe identity of Theorem~\ref{Concha-Lapointe identity} gives
     %$$\mathrm{sh}\left( P_{(\lambda|\gamma)}(X|y)^{\dagger} \right) = q^{|\gamma|}t^{\mathrm{inv}(\gamma) - \ell(w_0^{(k)})}w_0^{(k)}T_{w_0^{(k)}}P_{(\lambda|\gamma)}$$  
    \begin{align*}
    \Xi_k \left( \widetilde{E}_{(\gamma_1+1,\ldots \gamma_k+1|\lambda)} \right) 
    &= \frac{\prod_{i\ge 1}\prod_{j=1}^{m_i(\lambda)}( 1-t^{j})}{q^{\gamma_1+\cdots+\gamma_k}}\,\left(q^{\gamma_1+\dotsm+\gamma_k}t^{\mathrm{inv}(\gamma)-\ell(w_0^{(k)})}w_0^{(k)}T_{w_0^{(k)}}P_{(\lambda|\gamma)}\left(\frac{X}{t^{-1}-1}\Big| y\right)\right)^*\\
    &= \frac{\prod_{i\ge 1}\prod_{j=1}^{m_i(\lambda)}( 1-t^{j})}{t^{\ell(w_0^{(k)})-\mathrm{inv}(\gamma)}}\,\left(w_0^{(k)}T_{w_0^{(k)}}P_{(\lambda|\gamma)}\left(\frac{X}{t^{-1}-1}\Big| y\right)\right)^*\\
    &= \frac{\prod_{i\ge 1}\prod_{j=1}^{m_i(\lambda)}( 1-t^{j})}{j_{(\lambda|\gamma)}(q,t^{-1})t^{\ell(w_0^{(k)})-\mathrm{inv}(\gamma)}}\,\left(w_0^{(k)}T_{w_0^{(k)}}\mathcal{J}_{(\lambda|\gamma)}\left(\frac{X}{t^{-1}-1}\Big| y\right)\right)^*.\qedhere
    \end{align*}

\end{proof}

\section{Aligning Spectra}

\subsection{$z$-Action}

We recall the definition of the $\zz$ operators on the $V_{\bullet}$ space.

\begin{defn}[{Carlsson--Gorsky--Mellit \cite{GCM_2017}}]
    Define the operators $\zz_1,\ldots ,\zz_k: V_{k} \rightarrow V_{k}$ by 
    $$\mathsf{T}_1^{-1}\cdots \mathsf{T}_{k-1}^{-1}\zz_k F:= \frac{1}{t^{-1}-1}[\dd_{+}^{*},\dd_{-}]$$ and for $1 \leq i \leq k-1$
    $$\zz_{i}:= t\mathsf{T}_i^{-1}\zz_{i+1}\mathsf{T}_{i}^{-1}$$ where 
    $$\mathsf{T}_iF := \frac{(t-1)y_i}{y_{i+1}-y_i}F +\frac{y_{i+1}-ty_i}{y_{i+1}-y_{i}}s_i(F)$$ and $\dd_{+}^{*}$ is the operator on the larger spaces $\frac{1}{y_1\cdots y_k} V_k \rightarrow \frac{1}{y_1\cdots y_{k+1}} V_{k+1}$ given by 
    $$\dd_{+}^{*}\left(\frac{1}{y_1\cdots y_{k}} F(X|y_1,\ldots,y_k) \right):= \frac{-q^{-1}t^{-1}}{y_1\cdots y_{k+1}} F(X+q(t-1)y_{1}|y_2,\ldots,y_k).$$
\end{defn}
Note that even though the operator $\dd_{+}^{*}$ does not map $V_{k}\rightarrow V_{k+1}$, the commutator $[\dd_{+}^{*},\dd_{-}]$ does give a well defined map $V_{k} \rightarrow V_k$. Therefore, the operators $\zz_1,\ldots, \zz_k$ are well defined as well. We have corrected a minor sign error from \cite{GCM_2017} in the below formula.

\begin{prop}[\cite{GCM_2017}] For any $F=F(X\mid y_1,\dotsc,y_k)\in V_k$,
     $$\mathsf{T}_1^{-1}\cdots \mathsf{T}_{k-1}^{-1}\zz_k F= F\left(X+(t-1)qy_1-(t-1)u\mid y_2,\ldots,y_k,u\right)\Omega(u^{-1}qy_1-u^{-1}X)|_{u^{0}}$$
     where $\Omega$ is the plethystic exponential 
    $$\Omega(X) := \sum_{n \geq 0} h_{n}(X)$$
    and $|_{u^i}$ is the operator of taking the coefficient of $u^i$.
\end{prop}

%\begin{proof}
%First,
   % \begin{align*}
        %d_{+}^{*}d_{-}F(X|y_1,\ldots,y_k) &= d_{+}^{*} \left(-F(X-(t-1)y_k|y_1,\ldots,y_{k-1},u)\Omega(-u^{-1}X)|_{u^{-1}} \right)\\
        %&= (qty_1)^{-1}\gamma F(X+(t-1)y_k -(t-1)u|y_1,\ldots,y_{k-1},u)\Omega(-u^{-1}(X+(t-1)y_k))|_{u^{-1}}\\
        %&= (qty_1)^{-1}F(X+(t-1)qy_1-(t-1)u|y_2,\ldots,y_k,u)\Omega(-u^{-1}(X+(t-1)qy_1))|_{u^{-1}}.\\
   %\end{align*}

%Second,
    %\begin{align*}
        %d_{-}d_{+}^{*}F(X|y_1,\ldots,y_k) &= -d_{-}(qty_1)^{-1}\gamma F(X+(t-1)y_{k+1}|y_1,\ldots,y_k)\\
        %&= -d_{-}(qty_1)^{-1}F(X+(t-1)qy_1|y_2,\ldots,y_{k+1})\\
        %&= (qty_1)^{-1}F(X-(t-1)u+(t-1)qy_1|y_2,\ldots,y_k,u) \Omega(-u^{-1}X)|_{u^{-1}}.\\
    %\end{align*}
%Putting this together:
    %\begin{align*}
    %&[d_{+}^{*},d_{-}]F(X|y_1,\ldots,y_k) \\
    %&= (qty_1)^{-1} F(X+(t-1)qy_1-(t-1)u|y_2,\ldots,y_k,u) \left(\Omega(-u^{-1}(X+(t-1)qy_1)) - \Omega(-u^{-1}X) \right) |_{u^{-1}} \\
    %&= (qty_1)^{-1} F(X+(t-1)qy_1-(t-1)u|y_2,\ldots,y_k,u) \Omega(-u^{-1}X) \left( \Omega((1-t)u^{-1}qy_1) -1\right)|_{u^{-1}}. \\
    %\end{align*}

    %A simple calculation shows that 
    %$$\Omega((1-t)u^{-1}qy_1) -1 = (1-t)u^{-1}qy_1\Omega(u^{-1}qy_1).$$

    %Therefore, 
    %\begin{align*}
        %&[d_{+}^{*},d_{-}]F(X|y_1,\ldots,y_k) \\
        %&= (qty_1)^{-1} F(X+(t-1)qy_1-(t-1)u|y_2,\ldots,y_k,u) \Omega(-u^{-1}X) (1-t)u^{-1}qy_1\Omega(u^{-1}qy_1) |_{u^{-1}} \\
        %&= t^{-1}(1-t)u^{-1}F(X+(t-1)qy_1-(t-1)u|y_2,\ldots,y_k,u)\Omega(-u^{-1}X)\Omega(u^{-1}qy_1) |_{u^{-1}}\\
        %&= (t^{-1}-1)F(X+(t-1)qy_1-(t-1)u|y_2,\ldots,y_k,u)\Omega(u^{-1}qy_1-u^{-1}X)|_{u^0}.\\
    %\end{align*}
    %Lastly, by dividing by $(t^{-1}-1)$ we obtain the result.
    
%\end{proof}

We require the following technical lemma.

\begin{lem}\label{hecke op lemma}
    $$\mathsf{T}_{m-1} \cdots \mathsf{T}_k(x_k^{b}) = h_b(x_m+(1-t)(x_k+\ldots + x_{m-1}))$$
\end{lem}
\begin{proof}
    Base cases: If $m = 0$ then 
    $$x_k^b = h_{b}(x_k).$$ If $m=1$ then using the fact that $h_{j}(-x_k) = (-1)^{j}e_j(x_k) = 0$ for all $j \geq 2$ we see that
    \begin{align*}
        \mathsf{T}_k(x_k^b) &= x_{k+1}^b +(1-t)(x_k^{b}+x_k^{b-1}x_{k+1}+\ldots + x_kx_{k+1}^{b-1})\\
        &= (x_k^{b}+x_k^{b-1}x_{k+1}+\ldots + x_kx_{k+1}^{b-1}+ x_{k+1}^{b}) -tx_k(x_k^{b-1}+x_k^{b-2}x_{k+1}+\ldots + x_kx_{k+1}^{b-2}+x_{k+1}^{b-1})\\
        &= h_b(x_k+x_{k+1}) +h_{b-1}(x_k+x_{k+1})h_1(-tx_{k})\\
        &= \sum_{i = 0}^{b} h_i(x_k+x_{k+1})h_{b-i}(-tx_k)\\
        &= h_{b}(x_k+x_{k+1}-tx_k)\\
        &= h_b(x_{k+1}+(1-t)x_k).
    \end{align*}

    Now by induction we see:

    \begin{align*}
        \mathsf{T}_{m+1}(\mathsf{T}_{m}\cdots \mathsf{T}_k(x_k^b)) &= \mathsf{T}_{m+1}(h_b[x_{m+1}+(1-t)(x_k+\ldots + x_m)])\\
        &= \mathsf{T}_{m+1}\left( \sum_{i=0}^{b}x_{m+1}^{i}h_{b-i}[(1-t)(x_k+\ldots x_m)] \right)\\
        &= \sum_{i=0}^{b} h_i[x_{m+2}+(1-t)x_{m+1}]h_{b-i}[(1-t)(x_k+\ldots + x_{m})]\\
        &= h_b[x_{m+2}+(1-t)x_{m+1}+(1-t)(x_k+\ldots + x_{m})]\\
        &= h_b[x_{m+2}+(1-t)(x_k+\ldots + x_{m+1})].\qedhere
    \end{align*}
\end{proof}

The following result is implicit in \cite{Ion_2022} but we include its proof for the sake of completeness. Recall that $\widetilde{\lim}_{m}$ denotes the $t$-adic stable-limit as defined Ion--Wu \cite{Ion_2022}. 

\begin{prop}[\cite{Ion_2022}]\label{prop: limit cherednik action explicit}
For all $k \geq 0$ and $F(\mathfrak{X}_k|x_1,\ldots,x_k) \in \K[x_1,\ldots,x_k]\otimes \Lambda(\mathfrak{X}_k),$
\begin{align*}
    &\mathscr{Y}_1\mathsf{T}_{1}\cdots \mathsf{T}_{k-1}(x_1\cdots x_k F(\mathfrak{X}_k|x_1,\ldots,x_k)) \\
    &= qt^k x_1\cdots x_k F(\mathfrak{X}_k+ qx_1-u|x_2,\ldots,x_k,u) \Omega
(u^{-1}(qx_1+(1-t)\mathfrak{X}_k))|_{u^0}.\\
\end{align*}

\end{prop}
\begin{proof}
In what follows we let $W = w_1+w_2+\ldots$ be another set of free variables. The idea in the following calculation is to evaluate the action of $\mathscr{Y}_1\mathsf{T}_{1}\cdots \mathsf{T}_{k-1}$ on a sufficiently general element $x_1^{a_1+1}\cdots x_k^{a_k+1}\Omega(W\mathfrak{X}_k)$ of $V_k$ and therefore conclude the result. By direct computation using the above lemma \ref{hecke op lemma}:
    \begin{align*}
        &\mathscr{Y}_1\mathsf{T}_{1}\cdots \mathsf{T}_{k-1}(x_1^{a_1+1}\cdots x_k^{a_k+1}\Omega(W\mathfrak{X}_k))\\
        &= \widetilde{\lim}_{m} t^m \rho \omega_{m}^{-1}\mathsf{T}_{m-1}^{-1}\cdots \mathsf{T}_{k}^{-1}(x_1^{a_1+1}\cdots x_k^{a_k+1}\Omega(W(x_{k+1}+\ldots + x_m))) \\
        &= \widetilde{\lim}_{m} t^m \rho (qt^{k-m}X_1)\omega_{m}^{-1}\mathsf{T}_{m-1}^{-1}\cdots \mathsf{T}_{k}^{-1}(x_1^{a_1+1}\cdots x_k^{a_k}\Omega(W(x_{k+1}+\ldots +x_m)))\\
        &= \widetilde{\lim}_{m} t^m (qt^{k-m}X_1)\omega_{m}^{-1}\mathsf{T}_{m-1}\cdots \mathsf{T}_{k}(x_1^{a_1+1}\cdots x_k^{a_k}\Omega(W(x_{k+1}+\ldots +x_m))) \\
        %&= \widetilde{\lim}_{m}  qt^{k}X_1\omega_{m}^{-1}\mathsf{T}_{m-1}\cdots \mathsf{T}_{k}(x_1^{a_1+1}\cdots x_k^{a_k}\Omega(W(x_{k+1}+\ldots +x_m)))\\
        &= \widetilde{\lim}_{m}  qt^{k}X_1\omega_{m}^{-1}x_1^{a_1+1}\cdots x_{k-1}^{a_{k-1}+1}\mathsf{T}_{m-1}\cdots \mathsf{T}_{k}(x_k^{a_k}\Omega(W(x_{k+1}+\ldots +x_m))) \\
        &= \widetilde{\lim}_{m}  qt^{k}X_1X_2^{a_1+1}\cdots X_{k}^{a_{k-1}+1}\omega_{m}^{-1}\mathsf{T}_{m-1}\cdots \mathsf{T}_{k}(x_k^{a_k}\Omega(W(x_{k+1}+\ldots +x_m)))\\
        &= \widetilde{\lim}_{m}  qt^{k}X_1X_2^{a_1+1}\cdots X_{k}^{a_{k-1}+1}\omega_{m}^{-1}\Omega(W(X_{k}+\ldots + X_m))\mathsf{T}_{m-1}\cdots \mathsf{T}_{k}(x_k^{a_k}\Omega(-x_{k}W)) \\
        &= \widetilde{\lim}_{m}  qt^{k}X_1X_2^{a_1+1}\cdots X_{k}^{a_{k-1}+1}\Omega(W(qX_1+ X_{k+1}+\ldots + X_{m}))\omega_{m}^{-1}\mathsf{T}_{m-1}\cdots \mathsf{T}_{k}(x_k^{a_k}\Omega(-x_{k}W)) \\
        &= \widetilde{\lim}_{m} qt^{k}X_1X_2^{a_1+1}\cdots X_{k}^{a_{k-1}+1}\Omega(W(qX_1+ X_{k+1}+\ldots + X_{m})) \\
        &\times \omega_{m}^{-1}\left(\sum_{\ell \geq 0}(\mathsf{T}_{m-1}\cdots \mathsf{T}_{k})(x_k^{a_k+\ell})(-1)^{\ell}e_{\ell}(W) \right)\\
        &= \widetilde{\lim}_{m} qt^{k}X_1X_2^{a_1+1}\cdots X_{k}^{a_{k-1}+1}\Omega(W(qX_1+ X_{k+1}+\ldots + X_{m}))\\
        & \times \omega_{m}^{-1}\left(\sum_{\ell \geq 0}h_{a_k+\ell}(x_m+(1-t)(x_k+\ldots +x_{m-1}))(-1)^{\ell}e_{\ell}(W) \right)\\
        &= \widetilde{\lim}_{m} qt^{k}X_1X_2^{a_1+1}\cdots X_{k}^{a_{k-1}+1}\Omega(W(qX_1+ X_{k+1}+\ldots + X_{m}))\\
        & \times  \sum_{\ell \geq 0}h_{a_k+\ell}(qx_1+(1-t)(x_{k+1}+\ldots +x_{m}))(-1)^{\ell}e_{\ell}(W) \\
        &= qt^k x_1x_2^{a_1+1}\cdots x_{k}^{a_{k-1}+1} u^{a_k}\Omega(W(qx_1+\mathfrak{X}_k))\Omega(u^{-1}(qx_1+(1-t)\mathfrak{X}_k))\Omega(-uW)|_{u^{0}}\\
        &= qt^k x_1\cdots x_k x_2^{a_1}\cdots x_{k}^{a_{k-1}} u^{a_k} \Omega(W(\mathfrak{X}_k +qx_1-u))\Omega(u^{-1}(qx_1+(1-t)\mathfrak{X}_k))|_{u^{0}}.
    \end{align*}
    This implies the result.
\end{proof}

The following is a special case of a result of Ion--Wu \cite{Ion_2022}.

\begin{cor}[\cite{Ion_2022}]
For all $k \geq 0$ and $1 \leq i \leq k$
    $$\Xi_{k} \mathscr{Y}_i = qt\zz_i \Xi_{k}.$$
\end{cor}
\begin{proof}
    On one hand:
    \begin{align*}
        &\Xi_{k} \mathscr{Y}_1\mathsf{T}_{1}\cdots \mathsf{T}_{k-1}(x_1\cdots x_k F(\mathfrak{X}_k|x_1,\ldots,x_k)) \\
        &= \Xi_{\bullet} qt^k x_1\cdots x_k F(\mathfrak{X}_k+ qx_1-u|x_2,\ldots,x_k,u) \Omega(u^{-1}(qx_1+(1-t)\mathfrak{X}_k))|_{u^0} \\
        &= qt^k F \left(  \frac{X}{t-1} + qy_1-u | y_2,\ldots,y_k,u \right)\Omega(u^{-1}(qy_1-X))|_{u^0}.\\
    \end{align*}
    On the other hand:
    \begin{align*}
        & \mathsf{T}_1^{-1}\cdots \mathsf{T}_{k-1}^{-1}\zz_k \Xi_{k}(x_1\cdots x_kF(\mathfrak{X}_k|x_1,\ldots,x_k)) \\
        &= \mathsf{T}_1^{-1}\cdots \mathsf{T}_{k-1}^{-1}\zz_k F \left(\frac{X}{t-1}\Big|y_1,\ldots,y_k\right) \\
        &=  F\left(\frac{X+(t-1)qy_1-(t-1)u}{t-1} \Big|y_2,\ldots,y_k,u \right) \Omega(u^{-1}qy_1-u^{-1}X)|_{u^0}\\
        &= F\left(\frac{X}{t-1}+qy_1-u \Big|y_2,\ldots,y_k,u \right) \Omega(u^{-1}(qy_1-X))|_{u^0}.\\
    \end{align*}
    Therefore, 
    $$\Xi_{k} \mathscr{Y}_1\mathsf{T}_{1}\cdots \mathsf{T}_{k-1} = qt^k \mathsf{T}_1^{-1}\cdots \mathsf{T}_{k-1}^{-1}\zz_k \Xi_{k}.$$ But using the affine Hecke algebras relations
    $$qt^k \mathsf{T}_1^{-1}\cdots \mathsf{T}_{k-1}^{-1}\zz_k = qt\zz_1 \mathsf{T}_1\cdots \mathsf{T}_{k-1}.$$ Therefore, since 
    $$\Xi_{k}\mathsf{T}_i = \mathsf{T}_i\Xi_{k}$$ for all $1 \leq i \leq k-1 $ we find that 
    $$\Xi_{\bullet} \mathscr{Y}_1 = qt\zz_1 \Xi_{\bullet}$$ which implies 
    $$\Xi_{\bullet} \mathscr{Y}_j = qt\zz_j \Xi_{\bullet}$$ for all $ 1\leq j \leq k.$
\end{proof}

We may explicitly compute the action of the $\zz_i$ operators on both sides of \ref{eq eigenvectors} from Proposition \ref{Calculations of eigenvectors}.

\begin{cor}\label{eigenvector result}
For all $\lambda \in \mathbb{Y}$, $\gamma \in \mathbb{Z}_{\geq 0}^k$, and $1 \leq i \leq k$,
    \begin{align*}
    &\zz_i \left( w_0^{(k)} T_{w_0^{(k)}}J_{(\lambda|\gamma)}\left( \frac{X}{t^{-1}-1}\Big| y \right) \right)^{*}\\
    &= q^{\gamma_i}t^{\#(j<i|\gamma_j > \gamma_i)+ \#(i<j|\gamma_i \leq \gamma_j) + \#(j|\gamma_i < \lambda_j)} \left( w_0^{(k)} T_{w_0^{(k)}}J_{(\lambda|\gamma)}\left( \frac{X}{t^{-1}-1}\Big| y \right) \right)^{*}.
    \end{align*}
\end{cor}
\begin{proof}
    From Proposition \ref{Calculations of eigenvectors} we know that 
    $$\Xi_{k}(\widetilde{E}_{(\gamma_1+1,\ldots,\gamma_k+1|\lambda)}) = c_{(\lambda|\gamma)}\left( w_0^{(k)} T_{w_0^{(k)}}J_{(\lambda|\gamma)}\left( \frac{X}{t^{-1}-1}\Big| y \right) \right)^{*}$$ where 
    $$c_{(\lambda|\gamma)} = j_{(\lambda|\gamma)}(q,t^{-1})^{-1}t^{\ell(w_0^{(k)})-\mathrm{inv}(\gamma)}\prod_{i}\prod_{j=1}^{m_i(\lambda)}( 1-t^{j}).$$
    From \cite[Cor. 38]{MBWArxiv} we know that $\widetilde{E}_{(\gamma_1+1,\ldots,\gamma_k+1|\lambda)}$ is a $\mathscr{Y}_1,\ldots , \mathscr{Y}_k$ eigenvector with 
    $$\mathscr{Y}_i(\widetilde{E}_{(\gamma_1+1,\ldots,\gamma_k+1|\lambda)}) = q^{\gamma_i+1}t^{\ell(\gamma)+\ell(\lambda)+1-\beta_{(\gamma_1+1,\ldots,\gamma_k+1,\lambda)}(i)}\widetilde{E}_{(\gamma_1+1,\ldots,\gamma_k+1|\lambda)}$$ where for any composition $\alpha$
    $$\beta_{\alpha}(i)= \#(j \leq i| \alpha_j \leq \alpha_i)+\#(i<j| \alpha_i > \alpha_j).$$
    Therefore, 
    \begin{align*}
        & qt \zz_i \left( w_0^{(k)} T_{w_0^{(k)}}J_{(\lambda|\gamma)}\left( \frac{X}{t^{-1}-1}\Big| y \right) \right)^{*} \\
        &= qt\zz_i c_{(\lambda|\gamma)}^{-1}\Xi_{k}(\widetilde{E}_{(\gamma_1,\ldots,\gamma_k|\lambda)})\\
        &= c_{(\lambda|\gamma)}^{-1} qt\zz_i\Xi_{k}(\widetilde{E}_{(\gamma_1,\ldots,\gamma_k|\lambda)}) \\
        &= c_{(\lambda|\gamma)}^{-1} \Xi_{k}\mathscr{Y}_i(\widetilde{E}_{(\gamma_1,\ldots,\gamma_k|\lambda)})\\
        &= c_{(\lambda|\gamma)}^{-1}\Xi_{k}(q^{\gamma_i+1}t^{\ell(\gamma)+\ell(\lambda)+1-\beta_{(\gamma_1+1,\ldots, \gamma_k+1,\lambda)}(i)}\widetilde{E}_{(\gamma_1,\ldots,\gamma_k|\lambda)})\\
        &= q^{\gamma_i+1}t^{\ell(\gamma)+\ell(\lambda)+1-\beta_{(\gamma_1+1,\ldots, \gamma_k+1,\lambda)}(i)} c_{(\lambda|\gamma)}^{-1}\Xi_{k}(\widetilde{E}_{(\gamma_1,\ldots,\gamma_k|\lambda)})\\
        &= q^{\gamma_i+1}t^{\ell(\gamma)+\ell(\lambda)+1-\beta_{(\gamma_1+1,\ldots, \gamma_k+1,\lambda)}(i)} \left( w_0^{(k)} T_{w_0^{(k)}}J_{(\lambda|\gamma)}\left( \frac{X}{t^{-1}-1}\Big| y \right) \right)^{*}
    \end{align*}
    and by dividing both sides by $qt$ 
    \begin{align*}
    &\zz_i \left( w_0^{(k)} T_{w_0^{(k)}}J_{(\lambda|\gamma)}\left( \frac{X}{t^{-1}-1}\Big| y \right) \right)^{*}\\ &= q^{\gamma_i}t^{\ell(\gamma)+\ell(\lambda)-\beta_{(\gamma_1+1,\ldots, \gamma_k+1,\lambda)}(i)} \left( w_0^{(k)} T_{w_0^{(k)}}J_{(\lambda|\gamma)}\left( \frac{X}{t^{-1}-1}\Big| y \right) \right)^{*}.
    \end{align*}
    Lastly, by a simple calculation we find that 
    $$\ell(\gamma)+\ell(\lambda)-\beta_{(\gamma_1+1,\ldots, \gamma_k+1,\lambda)}(i) = \#(j<i|\gamma_j > \gamma_i)+ \#(i<j|\gamma_i \leq \gamma_j) + \#(j|\gamma_i < \lambda_j)$$ finishing the proof.
\end{proof}

Recall the definition of the modified partially symmetric Macdonald functions $\widetilde{H}_{(\lambda|\gamma)}(X|y)$ in Definition \ref{modified partially sym MacD function}.

\begin{thm}
    $\widetilde{H}_{(\lambda|\gamma)}$ is a $\zz_1,\ldots, \zz_k$ weight vector with 
    $$\zz_{i}\widetilde{H}_{(\lambda|\gamma)} = q^{\gamma_i}t^{\#(j| \lambda_j \geq \gamma_i +1) + \#(j>i|\gamma_j = \gamma_i)+\#(j|\gamma_j>\gamma_i)}\widetilde{H}_{(\lambda|\gamma)}.$$
\end{thm}
\begin{proof}
    Since $\widetilde{H}_{(\lambda|\gamma)}$ is a scalar multiple of 
    $$\left(w_0^{(k)} T_{w_0^{(k)}}\mathcal{J}_{(\lambda|\gamma)}\left( \frac{X}{t^{-1}-1}\Big| y \right) \right)^{*}$$ we know from Corollary \ref{eigenvector result} that $\widetilde{H}_{(\lambda|\gamma)}$ is a $\zz_1,\ldots,\zz_k$ weight vector with 
    $$\zz_{i}\widetilde{H}_{(\lambda|\gamma)} = q^{\gamma_i}t^{\#(j<i|\gamma_j > \gamma_i)+ \#(i<j|\gamma_i \leq \gamma_j) + \#(j|\gamma_i < \lambda_j)}\widetilde{H}_{(\lambda|\gamma)}.$$ It is easy to verify that 
    \begin{align*}
    \#(j<i|\gamma_j > \gamma_i)+ \#(i<j|\gamma_i \leq \gamma_j) + \#(j|\gamma_i < \lambda_j) \\= \#(j| \lambda_j \geq \gamma_i +1) + \#(j>i|\gamma_j = \gamma_i)+\#(j|\gamma_j>\gamma_i)
    \end{align*}
    completing the proof.
\end{proof}

We conclude this section with the following:

\begin{cor}
    Whenever $\phi(\mu,w) = (\lambda|\gamma),$ $\widetilde{H}_{(\lambda|\gamma)} \in V_k$ has the same $\zz_1,\ldots , \zz_k$ weight as $H_{\mu,w} \in U_k.$ 
\end{cor}

\subsection{$\Delta$ Operators and Simple Spectrum}

González--Gorsky--Simental introduced the following extension $\mathbb{B}_{q,t}^{\text{ext}}$ of the algebra $\mathbb{B}_{q,t}$. Recall that we have swapped the roles of $q,t$ compared to \cite{gonzález2023calibrated} and \cite{GCM_2017}.

\begin{defn}[{González--Gorsky--Simental \cite{gonzález2023calibrated}}]\label{ext B qt relations}
    The algebra $\mathbb{B}_{t,q}^{\text{ext}}$ is generated by $\mathbb{B}_{t,q}$ along with loops at each vertex $k \geq 0$ labelled by $\Delta_{p_m}$ for all $m \geq 1$ such that 
    \begin{itemize}
        \item $[\Delta_{p_m},\Delta_{p_{\ell}}] = 0$
        \item $[\Delta_{p_m},\mathsf{T}_i] = [\Delta_{p_m},\zz_i] = [\Delta_{p_m},\dd_{-}] = 0$
        \item $[\Delta_{p_m}, \dd_{+}] = \zz_1^{m}\dd_{+}.$
    \end{itemize}
\end{defn}

The next result is a fundamental fact regarding the polynomial representation of $\mathbb{B}_{t,q}.$

\begin{lem}[\cite{gonzález2023calibrated}]
\label{B qt poly rep generation}
    The $\mathbb{B}_{t,q}$ polynomial representation $V_{\bullet}$ is generated by $1 \in V_{0}.$
\end{lem}

Using Lemma \ref{B qt poly rep generation}, González--Gorsky--Simental show that we may uniquely define an extended action of $\mathbb{B}_{t,q}^{\text{ext}}$ on $V_{\bullet}$ agreeing with the previous action of $\mathbb{B}_{t,q}.$ 

\begin{prop}[\cite{gonzález2023calibrated}] \label{uniqueness of delta operators}
    We can \textit{uniquely} define operators $\Delta_{p_m}$ for all $m\geq 1$ on the polynomial representation $V_{\bullet}$ by setting $\Delta_{p_m}(1) = 0$ and extending via Lemma \ref{B qt poly rep generation} using the relations in Definition \ref{ext B qt relations}.
\end{prop}

Proposition \ref{uniqueness of delta operators} endows the space $V_{\bullet}$ with the structure of a $\mathbb{B}_{t,q}^{\text{ext}}$-module. These operators correspond via the $\mathbb{B}_{t,q}$ isomorphism $\Phi^{-1}: V_{\bullet} \rightarrow U_{\bullet}$ to the operators given by 
\begin{equation}\label{eq: delta spectrum 1}
    \Delta_{p_m} H_{\mu,w} = B_{\mu}(t^m,q^m) H_{\mu,w}
\end{equation}
 where $B_{\mu}(t,q):= \sum_{(i,j) \in  \mu}t^iq^j$ is the diagram sum of $\mu.$ We could have equivalently defined operators $\Delta_{e_m}$ for all $m \geq 1$ since both $\{ p_m \}_{m\geq 1}$ and $\{ e_m \}_{m\geq 1}$ generate $\Lambda_{q,t}$ algebraically. The operators $\Delta_{e_m}$ correspond to the geometric operators on the localized K-theory of each $\PFH_{n,n-k}$ by the $m$-th exterior power $\Lambda^{m} \mathcal{V}$ of the tautological  bundle $\mathcal{V}$ whose fibers are $\mathcal{V}^{-1}([I^{(n)}\subset \ldots \subset  I^{(n-k)}]) = \mathbb{C}[x,y]/I^{(n)}.$ This means that 
$$\Delta_{e_m} H_{\mu,w} = e_{m}(B_{\mu}(t,q)) H_{\mu,w}. $$

We will now look at the corresponding family of $\Delta$-operators on the space of almost symmetric functions $\mathscr{P}_{as}^{+}.$

\begin{defn}[\cite{MBWArxiv}]
     For any $F \in \Lambda$ and $n \geq 1$ define the operator $\Psi^{(n)}_{F}$ on $\K[x_1,\ldots,x_n]$ as
     
     $$\Psi^{(n)}_{F}:= F( t^nY_{1}^{(n)},\ldots , t^nY_{n}^{(n)} ).$$

     For $\nu \in \mathbb{Y}$ we denote 
    $$\kappa_{\nu}(q,t):= \sum_{i=1}^{\infty} q^{\nu_i}t^{i} \in \K.$$
\end{defn}

The next result follows by using the properties of the Ion--Wu limit $\widetilde{\lim}_{n}$ along with detailed DAHA computations.

\begin{thm}[\cite{BW_Delta}]
    For any $F\in \Lambda$ the sequence of operators $\Psi_{F}^{(n)}$ converges to an operator $F(\Delta)$ on $\mathscr{P}_{as}^{+}$.
\end{thm}

Here we recall many of the properties of the $\Delta$-operators on $\mathscr{P}_{as}^{+}.$

\begin{thm}[\cite{BW_Delta}]\label{higher delta operators thm}
     The following properties hold:
    \begin{itemize}
        \item $F(\Delta)(\widetilde{E}_{(\mu|\lambda)}) = F(\kappa_{\sort(\mu*\lambda)}(q,t))\widetilde{E}_{(\mu|\lambda)}$
        \item $[F(\Delta),\mathscr{Y}_i] = 0$
        \item $[F(\Delta), \mathsf{T}_i] = 0$
        \item $[F(\Delta),G(\Delta)] = 0$
        \item $\widetilde{\pi}F(\Delta) = F(\Delta +(q^{-1}-1)\mathscr{Y}_1)\widetilde{\pi}$
        \item The joint $\mathscr{Y}$-$p_1(\Delta)$ spectrum on $\mathscr{P}_{as}^{+}$ is simple.
    \end{itemize}
\end{thm}

It will be convenient to renormalize these $\Delta$-operators in order to align their action with their geometric analogues.

\begin{defn}[\cite{BW_Delta}]
    For $m \geq 1$ define the operator $\Delta_{p_{m}}$ on $\mathscr{P}_{as}^{+}$ as
    $$\Delta_{p_{m}}:= \frac{1}{q^m-1}\left( t^{-m}p_{m}(\Delta) - \frac{1}{1-t^m} \right).$$
\end{defn}

 Importantly, the operators $F(\Delta)$, and therefore $\Delta_{p_m}$, restrict to operators on each subspace $\K[x_1,\ldots,x_k]^+\otimes \Lambda(\mathfrak{X}_k).$ Furthermore, these operators act diagonally on the $\widetilde{E}_{(\mu|\lambda)}$ basis.

 \begin{lem}[\cite{BW_Delta}]
 \begin{equation}\label{eq: delta spectrum 2}
     \Delta_{p_m}\widetilde{E}_{(\mu|\lambda)} = \left( \sum_{i \geq 1}\left(\frac{q^{m~\sort(\mu*\lambda)_i}-1}{q^m-1}\right)t^{m(i-1)} \right)\widetilde{E}_{(\mu|\lambda)}
 \end{equation}
    
 \end{lem}
 %\begin{proof}
     %Immediate from Theorem \ref{higher delta operators thm}.
% \end{proof}

 %\begin{lem}
     %$$\Xi_{k} \widetilde{\pi} = t^{-k} d_{+} \Xi_{k} $$
 %\end{lem}
 %\begin{proof}
 %This follows from
     %$$X_1\mathsf{T}_1^{-1}\cdots \mathsf{T}_{k}^{-1} = t^{-k}\mathsf{T}_1\cdots \mathsf{T}_{k} X_{k+1}.$$
 %\end{proof}

 The operators $\Delta_{p_m}$ on each subspace $\K[x_1,\ldots,x_k]^+\otimes \Lambda(\mathfrak{X}_k) \subset \mathscr{P}_{as}^{+}$ exactly intertwine the corresponding operators on each $V_{k}.$

\begin{thm}[\cite{BW_Delta}]
For all $k \geq 0$
    $$\Xi_k \Delta_{p_m} = \Delta_{p_{m}} \Xi_k.$$
\end{thm}

%\begin{proof}
 %By Theorem \ref{higher delta operators thm} we see that as operators on $\mathscr{P}_{as}^{+}$
     %\begin{itemize}
        %\item $[\Delta_{p_m},\Delta_{p_{\ell}}] = 0$
        %\item $[\Delta_{p_m},\mathsf{T}_i] = [\Delta_{p_m},\mathscr{Y}_i] = [\Delta_{p_m},\partial_{-}] = 0$
        %\item $[\Delta_{p_m}, \widetilde{\pi}] = (q^{-1}t^{-1}\mathscr{Y}_1)^{m}\widetilde{\pi}.$
    %\end{itemize}
    %Therefore, the operators $\Delta_{p_m}':= \Xi \Delta_{p_m} \Xi^{-1}$ on $V_{\bullet}$ satisfy the properties in Definition \ref{ext B qt relations}. It is easy to check additionally that $\Delta_{p_m}'(1) = 0$ where $1 \in V_{0}.$ Therefore, by Proposition \ref{uniqueness of delta operators}, $\Delta_{p_m}' = \Delta_{p_m}.$
%\end{proof}

From here we immediately see the following result:

\begin{cor}[\cite{MBWArxiv},\cite{BW_Delta}]
    For all $ k \geq 0$ the  $\mathscr{Y}_1,\ldots, \mathscr{Y}_{k}, \Delta_{p_1}$ spectrum on $\K[x_1,\ldots,x_k]^+\otimes \Lambda(\mathfrak{X}_k)$ is simple. Furthermore, for all $m \geq 1$
    $$\Delta_{p_m}\widetilde{E}_{(\mu|\lambda)} = B_{\sort(\mu*\lambda)}(t^m,q^m) \widetilde{E}_{(\mu|\lambda)}. $$
\end{cor}

From \eqref{eq: delta spectrum 1} and \eqref{eq: delta spectrum 2} along with Corollary \ref{weights match thm}, we conclude the following result.

\begin{thm}\label{weights match thm}
    Whenever $\phi(\mu, w) = (\lambda|\gamma),$ $\widetilde{H}_{(\lambda|\gamma)} \in V_{k}$ has the same $\zz_1,\ldots,\zz_k,\Delta_{p_1}$ weight as $H_{\mu,w} \in U_{k}.$ 
\end{thm}

The fact that the joint $\Delta_{p_1},\zz_i$-weights of $H_{\mu,w}$ and $\widetilde{H}_{(\lambda|\gamma)}$ agree whenever $\phi(\mu, w) = (\lambda|\gamma)$ is sufficient to characterize the images $\Phi(H_{\mu,w})$ up to nonzero scalars.

\begin{cor}\label{cor:beta}
    There exist scalars $\beta_{\mu,w} \neq 0$ such that whenever $\phi(\mu,w) = (\lambda|\gamma)$, $$\Phi(H_{\mu,w}) = \beta_{\mu,w} \widetilde{H}_{(\lambda|\gamma)}.$$
\end{cor}
\begin{proof}

The map $\Phi$ is a $\mathbb{B}_{t,q}^{\text{ext}}$-module isomorphism. Therefore, we know that the images $\Phi(H_{\mu,w})$, which must be nonzero, are joint $\Delta_{p_1},\zz_i$-weight vectors. From Theorem \ref{weights match thm}, we know that whenever $\phi(\mu, w) = (\lambda|\gamma),$ $\widetilde{H}_{(\lambda|\gamma)}$ has the same $\zz_1,\ldots,\zz_k,\Delta_{p_1}$-weight. Since the joint $\Delta_{p_1},\zz_i$-weight spaces of $U_{\bullet}$ and $V_{\bullet}$ are simple, it must be that $\Phi(H_{\mu,w}) = \beta_{\mu,w}\widetilde{H}_{(\lambda|\gamma)}$  for some $\beta_{\mu,w} \neq 0$ whenever $\phi(\mu, w) = (\lambda|\gamma).$

\end{proof}

\section{Aligning $\dd_{-}$ and Main Theorem}

In this section, we will show that each of the coefficients $\beta_{\mu,w} = 1$ and in doing so prove our main result, Theorem~\ref{main thm}. We will do so by an inductive argument which will require the following base case:

\begin{lem}
For all $\mu \in \mathbb{Y},$
    $$\beta_{\mu,\varnothing} =1.$$
\end{lem}
\begin{proof}
    From Definition \ref{modified partially sym MacD function} we see that $\widetilde{H}_{(\mu|\varnothing)}= \widetilde{H}_{\mu}.$ The map $\Phi$ was shown by Carlsson--Gorsky--Mellit \cite{GCM_2017} to satisfy 
    $$\Phi(H_{\mu,\varnothing}) = \Phi(H_{\mu}) = \widetilde{H}_{\mu}.$$ Therefore, $\Phi(H_{\mu,\varnothing}) = \widetilde{H}_{(\mu|\varnothing)}$ so $\beta_{\mu,\varnothing} =1.$
\end{proof}

We will also require the following observation which says that forgetting a box label in the Carlsson--Gorsky--Mellit combinatorial model
%$(\mu,wx) \rightarrow (\mu,x)$ 
has a simple description on the other side of the bijection $\phi.$
\begin{lem}
    Suppose $\gamma_k +1 \geq \lambda_1$ and that $\phi(\mu,wx) = (\lambda|\gamma).$ Then 
    $$\phi(\mu,w) = (\gamma_k +1,\lambda|\gamma_1,\ldots,\gamma_{k-1}).$$
\end{lem}
\begin{proof}
This is clear from the explicit description of $\phi$ in \cite{GO}.
\end{proof}

We will now focus on the special case of $(\lambda|\gamma)$ with $\gamma_{k} +1 \geq \lambda_1$ in order to show 
\begin{align}\label{e:dd-tH}
\dd_{-} \widetilde{H}_{(\lambda|\gamma)} = \widetilde{H}_{(\gamma_k+1,\lambda|\gamma_1,\ldots, \gamma_{k-1})}.
\end{align}
We will then argue using intertwiner operators that $\beta_{\mu,w} =1$ for all $(\mu,w).$ If $\phi(\mu,w) = (\lambda|\gamma)$ we will write $\beta_{(\lambda|\gamma)}:= \beta_{\mu,w}.$ Recall the definition of the Ion--Wu analogues of the $\dd_{-}$ maps.

\begin{defn}[\cite{Ion_2022}]
    $$\partial_{-}(x_1^{a_1}\cdots x_{k}^{a_k}F(\mathfrak{X}_{k})):= x_1^{a_1}\cdots x_{k-1}^{a_{k-1}} z^{a_k}F(\mathfrak{X}_{k-1} -z) \Omega((1-t)z^{-1}\mathfrak{X}_{k-1}) \Big|_{z^{0}}$$
\end{defn}

The operators $\partial_{-}$ interact well with the $\widetilde{E}_{(\mu|\lambda)}$ basis of $\mathscr{P}_{as}^{+}.$

\begin{prop}[\cite{MBWArxiv}]\label{prop:del-}
    If $\mu_{k} \geq \lambda_1$ then 
    $$\partial_{-} \widetilde{E}_{(\mu_1,\ldots, \mu_k|\lambda_1,\ldots, \lambda_{\ell})} = \widetilde{E}_{(\mu_1,\ldots, \mu_{k-1}|\mu_k,\lambda_1,\ldots, \lambda_{\ell})}$$
\end{prop}

The two versions of $\dd_{-}$ agree via the map $\Xi$ up to a sign. A version of this result appears in the work of Ion--Wu \cite{Ion_2022}.

\begin{lem}\label{lem:deld}
    $\Xi_k\partial_{-} = -\dd_{-}\Xi_k$
\end{lem}
\begin{proof}
    \begin{align*}
        &\Xi_k \partial_{-} (x_1^{a_1+1}\cdots x_{k}^{a_k+1}F(\mathfrak{X}_{k})) \\
        &= \Xi_k \left(x_1^{a_1+1}\cdots x_{k-1}^{a_{k-1}+1} z^{a_k+1}F(\mathfrak{X}_{k-1} -z) \Omega((1-t)z^{-1}\mathfrak{X}_{k-1}) \Big|_{z^{0}} \right)\\
        &= y_1^{a_1}\cdots y_{k-1}^{a_{k-1}} z^{a_k+1}F\left(\frac{X}{t-1} -z\right) \Omega(-z^{-1}X) \Big|_{z^{0}}\\
        &= y_1^{a_1}\cdots y_{k-1}^{a_{k-1}} z^{a_k}F\left(\frac{X-(t-1)z}{t-1} \right) \Omega(-z^{-1}X) \Big|_{z^{-1}}\\
        &= - \dd_{-} \left(y_1^{a_1}\cdots y_{k-1}^{a_{k-1}} y_k^{a_k}F\left(\frac{X}{t-1}\right)  \right)\\
        &= - \dd_{-}\Xi_{k}\left( x_1^{a_1+1}\cdots x_{k}^{a_k+1}F(\mathfrak{X}_{k}) \right).\qedhere\\
    \end{align*}
\end{proof}

By combining Definition \ref{modified partially sym MacD function} and Proposition \ref{Calculations of eigenvectors}, at this point we know that 
\begin{equation}\label{eqn:HgtE}
\widetilde{H}_{(\lambda|\gamma)}(X|y) = g_{(\lambda|\gamma)}(q,t) \Xi_{k}(\widetilde{E}_{(\gamma_1+1,\ldots,\gamma_k+1|\lambda)})
\end{equation}
where
$$
g_{(\lambda|\gamma)}(q,t):= j_{(\lambda|\gamma)} (q,t^{-1})t^{\textrm{inv}(\gamma)+ n(\sort(\lambda,\gamma))+|(\lambda|\gamma)|}\prod_{i\ge 1}\prod_{j=1}^{m_i(\lambda)}( 1-t^{j})^{-1}.$$

\begin{lem}\label{d- lemma}
Suppose $\gamma_k + 1 \geq \lambda_1$ and set $\gamma':= (\gamma_1,\ldots, \gamma_{k-1}).$ Then
    
    \begin{equation}\label{eq d- coefficient}
        \dd_{-}\widetilde{H}_{(\lambda|\gamma)}(X|y) = -\frac{g_{(\lambda|\gamma)}(q,t)}{g_{(\gamma_k+1,\lambda|\gamma')}(q,t)}\widetilde{H}_{(\gamma_k +1,\lambda|\gamma')}(X|y).
    \end{equation}
\end{lem}
\begin{proof}
This is immediate from Proposition~\ref{prop:del-}, Lemma~\ref{lem:deld}, and \eqref{eqn:HgtE}.
%We use $\Xi_k$ to translate the corresponding assertion for $\widetilde{E}_{(\ga_1+1,\dotsc,\ga_k+1|\la)}$:
%    \begin{align*}
%        d_{-}\widetilde{H}_{(\lambda|\gamma)}(X|y)&= g_{(\lambda|\gamma)}(q,t)d_{-} \Xi_{k} \widetilde{E}_{(\gamma_1+1,\ldots,\gamma_k+1|\lambda)} \\
%        &= -g_{(\lambda|\gamma)}(q,t) \Xi_{k} \partial_{-}\widetilde{E}_{(\gamma_1+1,\ldots,\gamma_k+1|\lambda)}\\
%        &= -g_{(\lambda|\gamma)}(q,t) \Xi_{k} \widetilde{E}_{(\gamma_1+1,\ldots,\gamma_{k-1}+1|\gamma_{k}+1,\lambda)}\\
%        &=  -\frac{g_{(\lambda|\gamma)}(q,t)}{g_{(\gamma_k+1,\lambda|\gamma')}(q,t)}\widetilde{H}_{(\gamma_k +1,\lambda|\gamma')}(X|y).\qedhere
%    \end{align*}
\end{proof}

Recall the definition of $j_{(\la|\ga)}(q,t)$ from Definition~\ref{def:j}. 
%$$j_{(\lambda|\gamma)}(q,t) = \prod_{\square \in \lambda^{-}}(1-q^{\ell(\square)}t^{\tilde{a}(\square)+1}) \prod_{\square \in \gamma}(1-q^{\ell(\square)+1}t^{a(\square)+1}).$$

\begin{lem}\label{arm-leg lemma}
    Whenever $\gamma_k + 1 \geq \lambda_1$, $$\frac{j_{(\lambda|\gamma_1,\ldots,\gamma_k)}(q,t)}{j_{(\gamma_k+1,\lambda|\gamma_1,\ldots,\gamma_{k-1})}(q,t)} = (1-t^{m_{\gamma_k +1}(\lambda)+1})^{-1}.$$
\end{lem}
\begin{proof}
    Let $\nu = (\lambda_{-},\gamma)$ and $\nu' = (\lambda_{-},\gamma_k+1,\gamma_1,\ldots, \gamma_{k-1}).$
    Each box contributes the same factor $1-q^at^b$ in both diagrams, except for the new box which has $\ell_{\nu'}(\square) = 0$ and $\widetilde{a}_{\nu'}(\square) = m_{\gamma_k +1}(\lambda)$ and contributes the factor $(1-q^{0}t^{m_{\gamma_k +1}(\lambda)+1}).$ The equality of the other factors is immediate for all boxes except those below the new box in $\nu'$. If $\square=(i,j)$ is such a box, one has $a_{\nu}(\square) = \widetilde{a}_{\nu'}(\square)$, since $j\le \gamma_r \le\gamma_{k}$ is equivalent to $j\le \gamma_r<\gamma_k+1$, and clearly $\ell_{\nu'}(\square) = \ell_{\nu}(\square)+1$; this results in the same factor since these boxes were in the nonsymmetric part of $\nu$ and are now in the symmetric part of $\nu'.$
\end{proof}

Using a combinatorial argument, we can simply \eqref{eq d- coefficient} from Lemma \ref{d- lemma}.

\begin{prop}\label{beta lemma}
Whenever $\gamma_{k}+1 \geq \lambda_1$,
    $$-\frac{g_{(\lambda|\gamma)}(q,t)}{g_{(\gamma_k+1,\lambda|\gamma_1,\ldots , \gamma_{k-1})}(q,t)} = 1.$$
\end{prop}
\begin{proof}
Let $\ga'=(\ga_1,\dotsc,\ga_{k-1})$. The following hold:
\begin{itemize}
    \item $\mathrm{inv}(\gamma') = \mathrm{inv}(\gamma) - \#\{1\leq i \leq k-1 | \gamma_i > \gamma_k \}$
    \item $|(\gamma_k+1,\lambda|\gamma')| = |(\lambda|\gamma)| + 1$
    \item $n(\textrm{sort}(\gamma_k+1,\lambda,\gamma')) = n(\textrm{sort}(\lambda|\gamma)) + \#\{\lambda_i = \gamma_k+1 \} + \# \{1\leq i \leq k-1 | \gamma_i > \gamma_k \}.$
\end{itemize}
Therefore, 
$$\frac{t^{\textrm{inv}(\gamma)+ n(\sort(\lambda,\gamma))+|(\lambda|\gamma)|}}{t^{\textrm{inv}(\gamma')+ n(\sort(\gamma_k+1,\lambda,\gamma'))+|(\gamma_k+1,\lambda|\gamma')|}} = \frac{t^{\textrm{inv}(\gamma)+ n(\sort(\lambda,\gamma))+|(\lambda|\gamma)|}}{t^{\textrm{inv}(\gamma)+ n(\sort(\lambda,\gamma))+|(\lambda|\gamma)| + \#\{\lambda_i = \gamma_k +1 \} +1}} = t^{-(m_{\gamma_k +1}(\lambda)+1)}.$$ Further, it is clear that 
$$\frac{\prod_{i\ge 1}\prod_{j=1}^{m_{i}(\lambda)}(1-t^j)^{-1}}{\prod_{i\ge 1}\prod_{j=1}^{m_{i}(\gamma_k+1,\lambda)}(1-t^j)^{-1}} = (1-t^{m_{\gamma_k+1} (\gamma_k+1,\lambda)})^{-1} = (1-t^{m_{\gamma_k+1} (\lambda)+1})^{-1}.$$
Lastly, using Lemma \ref{arm-leg lemma}, everything cancels exactly to give 
\begin{equation*}
-\frac{g_{(\lambda|\gamma)}(q,t)}{g_{(\gamma_k+1,\lambda|\gamma')}(q,t)} = - t^{-(m_{\gamma_k +1}(\lambda)+1)}\frac{1-t^{m_{\gamma_k +1}(\lambda)+1}}{1-t^{-(m_{\gamma_k +1}(\lambda)+1)}} = 1.\qedhere
\end{equation*}
\end{proof}

\begin{example}
    $$\frac{j_{(2,1|0,2)}}{j_{(3,2,1|0)}} = \frac{(1-t)(1-qt^2)(1-t)(1-q^2t^3)(1-qt^2)}{(1-t)(1-qt^2)(1-t)(1-q^2t^3)(1-qt^2)(1-t)} = (1-t)^{-1}$$
    $$-\frac{g_{(2,1|0,2)}}{g_{(3,2,1|0)}} = -\frac{t^9(1-t)^{-2}}{t^{10}(1-t)^{-3}} \times (1-t^{-1})^{-1} = 1$$
\end{example}

At this point, Lemma~\ref{d- lemma} and Proposition~\ref{beta lemma} establish \eqref{e:dd-tH} when $\ga_k+1\ge \la_1$.
Next we turn our attention to the scalars $\beta_{\mu,w}$ of Corollary~\ref{cor:beta}.

\begin{lem}\label{intertwiner lemma}
    If $\beta_{(\lambda|\gamma)} =1$, then 
    $\beta_{(\lambda|s_{i}(\gamma))} = 1$ for any $1\le i <k$.
\end{lem}
\begin{proof}
    If $\gamma_i = \gamma_{i+1}$, then the statement is trivial. Suppose that $\gamma_{i} > \gamma_{i+1},$ $\beta_{(\lambda|\gamma)} =1$ and let $\phi(\mu,w) = (\lambda|\gamma).$ Note that $\phi(\mu,s_i(w)) = (\lambda|s_i(\gamma)).$ As was argued in \cite{GO}, under the assumption $\gamma_{i} > \gamma_{i+1}$,
     $$\mathsf{T}_i \widetilde{H}_{(\lambda|\gamma)} = \frac{q^{-(\ell(u)+1)}t^{a(u)+1}-1}{q^{-(\ell(u)+1)}t^{a(u)}-1} \widetilde{H}_{(\lambda|s_i(\gamma))} + \frac{t-1}{q^{\ell(u)+1}t^{-a(u)}-1} \widetilde{H}_{(\lambda|\gamma)}$$ and 
     $$\mathsf{T}_{i}H_{\mu,w} = \frac{(t-1)w_{i+1}}{w_i-w_{i+1}}H_{\mu,w} + \frac{w_i-tw_{i+1}}{w_i-w_{i+1}} H_{\mu,s_i(w)}$$ where 
     $u = (n-k+i, \gamma_{i+1}+1) \in \text{dg}(\lambda^{-}|\gamma)$ and $n-k = |(\lambda|\gamma)|.$ Further, 
     $$\frac{w_i}{w_{i+1}} = q^{\ell(u)+1}t^{-a(u)}$$ so that 
     $$\mathsf{T}_i \widetilde{H}_{(\lambda|\gamma)} = \frac{(t-1)w_{i+1}}{w_i-w_{i+1}}\widetilde{H}_{(\lambda|\gamma)} + \frac{w_i-tw_{i+1}}{w_i-w_{i+1}} \widetilde{H}_{(\lambda|s_i(\gamma))}.$$ Since the map $\Phi$ is $\mathsf{T}_i$-equivariant we know that 
     \begin{align*}
         \Phi(\mathsf{T}_i H_{\mu,w}) &= \mathsf{T}_i \Phi(H_{\mu,w})\\
         &= \mathsf{T}_i \widetilde{H}_{(\lambda|\gamma)}\\
         &= \frac{(t-1)w_{i+1}}{w_i-w_{i+1}}\widetilde{H}_{(\lambda|\gamma)} + \frac{w_i-tw_{i+1}}{w_i-w_{i+1}} \widetilde{H}_{(\lambda|s_i(\gamma))}\\
     \end{align*}
     and 
     \begin{align*}
         \Phi (\mathsf{T}_i H_{\mu,w}) &= \Phi\left( \frac{(t-1)w_{i+1}}{w_i-w_{i+1}}H_{\mu,w} + \frac{w_i-tw_{i+1}}{w_i-w_{i+1}} H_{\mu,s_i(w)} \right)\\
         &= \frac{(t-1)w_{i+1}}{w_i-w_{i+1}}\widetilde{H}_{(\lambda|\gamma)} + \frac{w_i-tw_{i+1}}{w_i-w_{i+1}} \Phi(H_{\mu,s_i(w)}).\\
     \end{align*} 
     Since $\frac{w_i-tw_{i+1}}{w_i-w_{i+1}} \neq 0$ it follows that 
     $$\Phi(H_{\mu,s_i(w)}) = \widetilde{H}_{(\lambda|s_i(\gamma))}.$$ Thus $\beta_{\mu,s_i(w)} =1.$

    If $\gamma_{i}< \gamma_{i+1}$, then using the quadratic relation $(\mathsf{T}_{i}-1)(\mathsf{T}_i+t) = 0$ we may make a nearly identical argument to the above to show that $\beta_{(\lambda|s_i(\gamma))} =1.$
\end{proof}

\begin{prop}\label{scalars are all 1}
For all $(\mu,w)$,
    $$\beta_{\mu,w} = 1.$$
\end{prop}
\begin{proof}
     Consider any $(\nu|\gamma) = (\nu_1,\ldots,\nu_{\ell}|\gamma_1,\ldots,\gamma_k)$ and set $\lambda:= \sort(\nu,\gamma_1+1,\ldots,\gamma_{k}+1).$
    Using Proposition \ref{beta lemma} and Lemma \ref{d- lemma} we know by induction that
    $$\beta_{(\varnothing|\lambda_1-1,\ldots,\lambda_{\ell+k}-1)} = 1.$$ By repeatedly applying Lemma \ref{intertwiner lemma} we may swap entries to show 
    $$\beta_{(\varnothing|\gamma_1,\ldots,\gamma_k,\nu_1-1,\ldots,\nu_{\ell}-1)} =1.$$ Lastly, we inductively use Lemma \ref{d- lemma} again to see that 
    $\beta_{(\nu_1,\ldots,\nu_{\ell}|\gamma_1,\ldots,\gamma_k)} =1.$
\end{proof}

We immediately conclude the following:

\begin{thm}\label{main thm}
    Whenever $\phi(\mu,w) = (\lambda|\gamma)$, 
    $$\Phi(H_{\mu,w})= \widetilde{H}_{(\lambda|\gamma)}.$$
\end{thm}

\section{Some consequences}

\subsection{The $\mathcal{N}$ involution}

Define $\overline{\omega}:\Lambda \rightarrow \Lambda$ by $\overline{\omega}(F(X;q,t)):= F(-X;q^{-1},t^{-1}).$ Recall that the operator $\nabla:\Lambda \rightarrow \Lambda$, introduced by \cite{BGHT} and central to the Shuffle Theorem \cite{CM_2015}, is given on the modified Macdonald basis by 
$$\nabla(\widetilde{H}_{\mu}):= q^{n(\mu)}t^{n(\mu')}\widetilde{H}_{\mu}.$$ Carlsson--Gorsky--Mellit in \cite{GCM_2017} showed that there exists a $q,t$-antilinear operator $\mathcal{N}:V_{\bullet} \rightarrow V_{\bullet}$ satisfying the following properties:
\begin{enumerate}
    \item $\mathcal{N}$ maps each $V_k$ to $V_k$
    \item $\mathcal{N}^2 = 1$
    \item $\mathcal{N}|_{V_0} = \nabla \circ \overline{\omega}$
    \item $\mathcal{N}\dd_{-}\mathcal{N} = \dd_{-}$
    \item $\mathcal{N} \mathsf{T}_i \mathcal{N} = \mathsf{T}_{i}^{-1}$
    \item $\mathcal{N} \dd_{+} \mathcal{N} = t^{-k}\zz_1\dd_{+}.$
\end{enumerate}

%\begin{rem}  We remark that the operators $z_i$ on $V_{\bullet}$ are invertible (these are just the \textit{dual} line bundle operators geometrically) and that in fact $\mathcal{N}z_i\mathcal{N} = z_i^{-1}.$ If we extend $\mathbb{B}_{t,q}$ to an algebra, $\mathcal{B}$ say, by inverting the operators $z_i$, then the map $\mathcal{N}$ on $V_{\bullet}$ intertwines a $q,t$-antilinear isomorphism of the algebra $\mathcal{B}$ itself. This map $\iota: \mathcal{B} \rightarrow \mathcal{B}$ is determined by $\iota(d_{-}) = d_{-}, \iota(\mathsf{T}_i) = \mathsf{T}_{i}^{-1}, \iota(d_{+}) = t^{-k}z_1d_{+}$, and $\iota(z_i) = z_i^{-1}.$ This map is \textbf{not} the same as $\Theta$ from \cite{gonzález2023calibrated}. There is a similar statement for $\mathbb{B}_{t,q}^{\text{ext}}$ as well. \end{rem}

The corresponding geometric map $\mathcal{N}':= \Phi^{-1}\circ \mathcal{N} \circ \Phi:U_{\bullet} \rightarrow U_{\bullet}$ is given by $\mathcal{N}' = \mathscr{L}\circ \mathrm{SD}\circ \mathscr{L}^{-1}$ where 
\begin{enumerate}
    \item $\mathscr{L}:= \det \mathcal{V}$ is the determinant of the canonical bundle $\mathcal{V}$
    \item $\mathrm{SD}$ is Serre-duality.
\end{enumerate}
Explicitly, 
$$\mathcal{N'}\left(\sum_{(\mu,w)}a_{\mu,w}(q,t)H_{\mu,w}\right) = \sum_{(\mu,w)}a_{\mu,w}(q^{-1},t^{-1})H_{\mu,w}.$$

Using Theorem \ref{main thm}, we may describe the action of $\mathcal{N}$ on all of $V_{\bullet}$ directly. 

\begin{cor}\label{action of N}
    The involution $\mathcal{N}:V_{\bullet} \rightarrow V_{\bullet}$ fixes each $\widetilde{H}_{(\la|\ga)}$. In particular, it is given on arbitrary elements by 
    \begin{equation}\label{eq action of N}
    \mathcal{N}\left(\sum_{(\lambda|\gamma)}a_{(\lambda|\gamma)}(q,t)\widetilde{H}_{(\lambda|\gamma)}\right) = \sum_{(\lambda|\gamma)}a_{(\lambda|\gamma)}(q^{-1},t^{-1})\widetilde{H}_{(\lambda|\gamma)}.
    \end{equation}
\end{cor}

A priori, we would already have known the above statement if we had \textit{defined} $\widetilde{H}_{(\lambda|\gamma)} = \Phi(H_{\mu,w})$ whenever $\phi(\mu,w) = (\lambda|\gamma).$ Therefore, the content of Corollary \ref{action of N} is that we have an \textit{explicit} formula \eqref{Calculations of eigenvectors} for the modified partially symmetric Macdonald functions $\widetilde{H}_{(\lambda|\gamma)}$ and thus \eqref{eq action of N} gives an \textit{explicit} expression for the action of $\mathcal{N}.$

\subsection{Pieri formulas}

Via matrix coefficient calculations of \cite[Sec. 5]{GCM_2017}, our Theorem~\ref{main thm} enables one to recover the Pieri formula for multiplication by $e_1(X)$ on partially-symmetric Macdonald polynomials which was proved in \cite{goodberryarxiv} and connected to \cite{GCM_2017} in \cite{GO}. One can also now obtain additional Pieri-type formulas for multiplication by $y_1,\dotsc,y_k$ in this way.

\printbibliography

\end{document}